\documentclass{amsart}

\usepackage[utf8]{inputenc} 
\usepackage[T1]{fontenc} 
\usepackage{amsthm}
\usepackage{amssymb}
\usepackage{mathrsfs}
\usepackage{multicol}
\usepackage{xcolor}
\usepackage{tikz}
\usetikzlibrary{shapes.geometric}
\usetikzlibrary{fit}
\usetikzlibrary{cd}
\usetikzlibrary{arrows,topaths}
\tikzstyle{rect}=[rectangle, minimum width=3cm, minimum height=1cm, text centered, text width=3cm]
\tikzstyle{rect1}=[rectangle, minimum width=3cm, minimum height=1cm, text centered, text width=4cm]
\tikzstyle{rect2}=[rectangle, minimum width=3cm, minimum height=1cm, text centered, text width=5cm]

\usepackage{hyperref}
\hypersetup{
	colorlinks=true,
	linkcolor=blue,
	filecolor=magenta,      
	urlcolor=cyan,
}

\theoremstyle{plain}
\newtheorem{prop}{Proposition}
\newtheorem{thm}[prop]{Theorem}

\newtheorem{lem}[prop]{Lemma}

\newtheorem*{theoA}{Theorem A}

\theoremstyle{definition}
\newtheorem{defi}[prop]{Definition}

\newtheorem{hypo}{Hypothesis}

\newtheorem{notation}[prop]{\bf Notation}

\theoremstyle{remark}
\newtheorem{rem}[prop]{Remark}
\newtheorem{example}{Example}

\numberwithin{prop}{section}
\numberwithin{equation}{section}

\newcommand{\N}{\mathbb{N}}
\newcommand{\Z}{\mathbb{Z}}

\newcommand{\image}{\mathrm{im}}

\newcommand{\argu}{\hbox to 7truept{\hrulefill}}

\newcommand{\iid}{\mathrm{id}}
\newcommand{\RestrTo}[1]{|_{#1}}
\newcommand{\dom}{\mathrm{dom}}

\newcommand{\Stab}{\mathrm{Stab}}

\newcommand{\qgraph}{\!\setminus\!\!\!\setminus}
\newcommand{\be}{{\bf e}}

\newcommand{\ob}{\mathrm{Ob}}
\newcommand{\arr}{\mathrm{Arr}}

\newcommand{\caC}{\mathcal{C}}
\newcommand{\caD}{\mathcal{D}}
\newcommand{\caF}{\mathcal{F}}
\newcommand{\caT}{\mathcal{T}}
\newcommand{\caG}{\mathcal{G}}

\newcommand{\caH}{\mathcal{H}}
\newcommand{\caK}{\mathcal{K}}

\newcommand{\caP}{\mathcal{P}}
\newcommand{\caR}{\mathcal{R}}
\newcommand{\caS}{\mathcal{S}}

\newcommand{\frD}{\mathfrak{D}}

\newcommand{\frG}{\mathfrak{G}}

\newcommand{\eps}{\varepsilon}

\newcommand{\tivarphi}{\tilde{\varphi}}

\newcommand{\dist}{\mathrm{dist}}

\newcommand{\catname}[1]{{\normalfont\textbf{#1}}}
\newcommand{\Set}{\catname{Set}}

\newcommand{\Grp}{\catname{Grp}}
\newcommand{\Grpd}{\catname{Grpd}}
\newcommand{\PGam}{\catname{$\caP\Gamma$}}
\newcommand{\Graphs}{\catname{Graphs}}
\newcommand{\PB}{\catname{PB}}

\definecolor{aogreen}{rgb}{0.0, 0.5, 0.0}
\definecolor{applegreen}{rgb}{0.55, 0.71, 0.0}
\definecolor{darkspringgreen}{rgb}{0.09, 0.45, 0.27}
\definecolor{darkturquoise}{rgb}{0.0, 0.81, 0.82}
\definecolor{brandeisblue}{rgb}{0.0, 0.44, 1.0}
\definecolor{crimson}{rgb}{0.86, 0.08, 0.24}
\definecolor{capri}{rgb}{0.0, 0.75, 1.0}
\usepackage{graphicx} 
\graphicspath{{Figures/}}

\title{Bass-Serre theory for groupoids}

\author[G. dal Verme, T. Weigel]{Giulia dal Verme$^1$ and Thomas Weigel$^2$$^{*}$}

\address{$^{1}$ Dipartimento di Matematica e Applicazioni, Universit\`a degli Studi di Milano-Bicocca, Ed.~U5, Via R.Cozzi 55,
20125 Milano, Italy.}
\email{\textcolor[rgb]{0.00,0.00,0.84}{giulia.dalverme@unimib.it}}

\address{$^{2}$ Dipartimento di Matematica e Applicazioni, Universit\`a degli Studi di Milano-Bicocca, Ed.~U5, Via R.Cozzi 55,
20125 Milano, Italy.}
\email{\textcolor[rgb]{0.00,0.00,0.84}{thomas.weigel@unimib.it}}
\date{Received: xxxxxx; Revised: yyyyyy; Accepted: zzzzzz.
\newline \indent $^{*}$ Corresponding author}
\begin{document}
\setcounter{page}{1}

\maketitle

\begin{abstract}
	In this paper a Bass-Serre theory 
	in the groupoid setting is developed and a structure theorem (cf. Theorem \ref{thm:structure}) is established.
	Any groupoid action without inversion of edges on a forest induces a graph of groupoids,
	while any graph of groupoids satisfying certain hypothesis admits a canonical associated groupoid, called the fundamental groupoid, and a forest, called the Bass-Serre forest, such that the fundamental groupoid acts on the Bass–Serre forest. 
	The structure theorem states that these processes are mutually inverse.
\end{abstract}

\section{Introduction}
\label{s:intro}
Actions of groups on trees play a fundamental role in many areas of mathematics. 
A part of J-P. Serre’s extensive contribution to the theory was the introduction of graphs of groups in \cite{ser:trees}. The theory of graphs of groups was further developed by H. Bass in \cite{bass}, and is now known as Bass-Serre theory.

The well-known Bass-Serre theory gives a complete and satisfactory description of groups acting on trees via the structure theorem. 
A graph of groups consists of a connected graph $\Gamma$ together with a group for each vertex and edge of $\Gamma$, and group monomorphisms from each edge group to the adjacent vertex groups. Any group action (without inversion of edges) on a tree induces a graph of groups,
while any graph of groups has a canonical associated group, called the {\it fundamental group}, and a tree, called the {\it Bass–Serre tree}, such that the fundamental group acts on the Bass–Serre tree. 
The structure theorem says that these processes are mutually inverse, so that graphs of groups \lq\lq encode\rq\rq\,  group actions on trees.
Graphs of groups are now fundamental tools in geometric group theory and low-dimensional topology. The fundamental group of a graph of groups generalises two basic constructions in combinatorial group theory, namely free products with amalgamation and HNN extensions.
These constructions correspond in topology to taking a connected sum and adding a handle, respectively.

Groupoids are small categories in which every morphism is invertible (see \cite{brown:fromgrp}). 
From an algebraic point of view,
groupoids can be thought as a generalization of a group,
i.e., a groupoid is a set with partial multiplication on it that could
contain many identities.
Every groupoid defines a disjoint union of groups, and thus a choice of base points is part of the language. This leads to the definition of the {\it source} and {\it range} maps for a groupoid. Despite the analogies between groups and groupoids, generalizing structural results from groups to groupoids is in general a highly non-trivial task.
The purpose of this paper fits in this paradigm: our intention is to establish a Bass-Serre theory for groupoids.

The concept of a group action on a space was generalized to a groupoid action and it has applications to dynamical systems, representation theory and operator algebras. If groups can roughly be described as the set of symmetries of certain objects, then groupoids can be thought as the set of symmetries of fibered objects.

In Sections \ref{s:prel} and \ref{s:gpdactions} it is convenient to look at groupoids from a categorical point of view, while in Sections \ref{s:graphgpd}, \ref{s:gractforest}, \ref{s:fundgpd1} and \ref{s:structhm} we use an algebraic point of view. 
The collection $\caG^{(0)}$ of idempotent elements in a groupoid $\caG$ is called its 
{\it unit space}, since these are precisely the elements $x$ that satisfy $x\alpha=\alpha$
and $\beta x=\beta$ whenever these products are defined. 
When considering an action of $\caG$ on a graph $\Gamma$, 
this leads to a fibered structure of $\Gamma$ over $\caG^{(0)}$.

A graph of groupoids $\caG(\Gamma)$ is given by a connected graph $\Gamma$ together with a groupoid for each vertex and edge of $\Gamma$, and monomorphisms from each edge groupoid to the adjacent vertex groupoid. 
We will only work with graph of groupoids having discrete vertex and edge groupoids.
We associate to any graph of groupoids $\caG(\Gamma)$ a groupoid $\pi_1(\caG(\Gamma))$, called the {\it fundamental groupoid}, and a forest $X_{\caG(\Gamma)}$,
called the {\it Bass-Serre forest}, such that $\pi_1(\caG(\Gamma))$ acts on $X_{\caG(\Gamma)}$.
We will use an algebraic approach to define the
fundamental groupoid $\pi_1(\caG(\Gamma))$ by generators and relations (see \cite{hig64} for presentations of groupoids).

One of the main differences between the group and groupoid settings is the following:
in the classical setting, given a group action without inversion on a graph, one of 
the ingredients used to build a graph of groups is the quotient graph given by such an action;
in the groupoid context, there is no canonical graph associated to the action of a groupoid 
on a graph. 
In the groupoid framework the construction of a graph which works similar to the quotient graph in the group setting has been challenging and requires some technical arguments. 
Given a groupoid $\caG$ acting without inversion on a forest $F$,  we use the fibered structure of $F$ on $\caG^{(0)}$ and the equivalence relation $\caR_\caG$ defined by the action $\caG$ on $F$ to recover the vertex set $F^0$. 
We then define a {\it tree of representatives} of the action of $\caG$ on $F$.
Subsequently, we are able to associate a graph of groupoids to the action of $\caG$ on $F$ via the definition of a
{\it desingularization} of a groupoid action on a graph. 
As a consequence, we obtain the following structure theorem (cf. Theorem \ref{thm:structure}) in this setting.
\begin{theoA}
\label{th:A}
 Let $\caG$ be a groupoid acting without inversion of edges on a forest $F$ such that each fiber (with respect to the momentum map) of the forest is a tree.
 Then $\caG$ is isomorphic to the fundamental groupoid of the graph of groupoids defined by a desingularization of the action of $\caG$ on $F$.
\end{theoA}

In classical Bass-Serre theory, given a group action without inversion on a tree $G\curvearrowright T$, one defines the vertex group $G_x$, $x\in\Lambda^0$, where $\Lambda=G\backslash T$ is the quotient graph, to be the stabilizers of a vertex representative $x'\in p^{-1}(x)$, where $p\colon T\to G\backslash T$ is the natural projection. 
Given a groupoid $\caG$ acting on a forest $F$, we define the vertex groupoids to be the stabilizers of the partial sections given by a desingularization of the action of $\caG$ on $F$.
Moreover, as in the group context, the vertex groupoids are isomorphic by the conjugacy relation.

In case one considers graph of groupoids given by groupoids whose unit space is a singleton, 
one recovers the classical Bass-Serre theory.

\section{Preliminaries}
\label{s:prel}

\subsection{Graphs}
\label{ss:graphs}
The notion of graph we use in this paper is the one used by J-P. Serre in \cite{ser:trees}, apart from the fact that we allow the graph with empty vertex set. This graph will be called the {\it empty graph}.
A {\it graph} $\Gamma$ (in the sense of J-P. Serre) consists of a set of vertices $\Gamma^0$, 
a set of edges $\Gamma^1$, a terminus map $t\colon\Gamma^1\to\Gamma^0$, 
an origin map $o\colon\Gamma^1\to\Gamma^0$ 
and an edge-reversing map $\bar{}\,\colon\Gamma^1\to\Gamma^1$ satisfying
\begin{equation*}
\bar{e}\ne e, \quad \bar{\bar{e}}=e,\quad t(\bar{e})=o(e).
\end{equation*}
Such a graph $\Gamma$ can be viewed as an {\it unoriented} (or {\it undirected}) graph
in which each geometric edge is replaced by a pair of edges $e$ and $\bar{e}$.
	
A {\it subgraph} $\Lambda$ of $\Gamma$ consists of a non-empty subset $\Lambda^0\subseteq\Gamma^0$ and a subset $\Lambda^1\subseteq\Gamma^1$
such that $\overline{\Lambda^1}\subseteq\Lambda^1$,
$t(\Lambda^1)\subseteq\Lambda^0$ and $o(\Lambda^1)\subseteq \Lambda^0$.
	
An {\it orientation} $\Gamma_+^1\subseteq\Gamma^1$ of $\Gamma$ is a set of edges containing
exactly one edge from each pair $\{e,\bar{e}\}$, $e\in\Gamma^1$.

Let $\Gamma$ and $\Lambda$ be two graphs. A {\it graph homomorphism} is a pair of mappings
	$\psi=(\psi^0,\psi^1)$, where 
	$\psi^0\colon\Gamma^0\to\Lambda^0$ and $\psi^1\colon\Gamma^1\to\Lambda^1$, which commutes 
	with $t$, $o$ and $\,\bar{}$ , i.e., 
	\begin{equation*}
	\psi^0(o(e))=o(\psi^0(e)),\quad \psi^0(t(e))=t(\psi^0(e)), \quad \psi^1(\bar{e})=\overline{\psi^1(e)}.
	\end{equation*}
A homomorphism of graphs is called an {\it isomorphism} if $\psi^0$ and $\psi^1$ are bijective.
By definition, any subgraph $\Lambda$ of a graph $\Gamma$ defines a canonical injective homomorphism of graphs $\iota\colon\Lambda\to\Gamma$.

For $v\in\Gamma^0$ we call 
\begin{equation*}
\mathrm{st}_\Gamma(v)=\{e\in\Gamma^1\mid t(e)=v\}
\end{equation*}
the {\it star} of $v$ in $\Gamma$. The cardinality of $\mathrm{st}_\Gamma(v)$ is called the {\it valence} of $v$ in $\Gamma$.
We say that $\Gamma$ is {\it locally-finite} if each vertex has finite valence.
A locally-finite graph such that each vertex has valence $k$ will be called {\it $k$-regular}.\\
A vertex of $\Gamma$ is said to be {\it singular} if it has valence one, i.e.,
if it is the terminus (equivalently, origin) of a unique edge.
We say that $\Gamma$ is {\it nonsingular} if it has no singular vertices.

\subsection{Paths}
\label{ss:paths}
	A {\it path} of length $n$ in $\Gamma$ is either a vertex $v\in\Gamma^0$ (when $n=0$),
	or a sequence of edges $e_1\dots e_n$ with $o(e_i)=t(e_{i+1})$ for all $i=1,\dots,n-1$.
	For a path $p=e_1\dots e_n$ with $t(e_1)=v$ and $o(e_n)=w$ we say that $p$ is a path from $v$ to $w$.
	We put $t(p)=t(e_1)$ and $o(p)=o(e_n)$ and denote by $\ell(p)=n$ the length of the path $p$.
	If $e_{i+1}\ne \bar{e}_i$ for all $i=1,\dots, n-1$, then we say that $p$ is {\it without backtracking}.
	A path of length $n$ is said to be {\it reduced} if either $n=0$ or if there is no backtracking.
	We denote by $\caP_{v,w}$  the set of all paths from $v$ to $w$ without backtracking.
	A {\it cycle} is a path with origin equal to its terminus.
	We say that the non-empty graph $\Gamma$ is {\it connected} if $\caP_{v,w}\ne\emptyset$ for all $v,w\in\Gamma^0$.
	For a path $p=e_1\dots e_n$, we define the {\it reversal} path $\bar{p}=\bar{e}_n\cdots\bar{e}_1$.

 In a connected graph $\Gamma$ one defines a distance function 
$\mathrm{dist}_\Gamma\colon \Gamma^0\times\Gamma^0\to\N_0$ by 
\begin{equation}
\mathrm{dist}_\Gamma(v,w)=\min\{\,\ell(p)\mid p \in\caP_{v,w}\,\},\quad v,w,\in\Gamma^0,
\end{equation}
which satisfies $\mathrm{dist}_\Gamma(v,w)=0$ if and only if $v=w$.\\

\begin{defi}
	A graph $\Gamma$ is said to be a {\it tree} if it is connected and for every $v\in\Gamma^0$, 
	the only reduced path which starts and ends at $v$ is the path of length $0$ at $v$.
A {\it forest} is a graph whose connected components are trees.
\end{defi}
Equivalently, $\Gamma$ is a tree if $|\caP_{v,w}|=1$ for all $v,w\in\Gamma^0$.

\begin{rem}
    Let $\Gamma$ be a connected graph. The set $\operatorname{SubTr}(\Gamma)$ of subgraphs of $\Gamma$ which are trees, ordered by inclusion, is directed and non-empty. 
    If $\Omega\subseteq\operatorname{SubTr}(\Gamma)$ is a totally ordered subset, 
    then $\bigcup_{T\in\Omega} T$ is a subtree of $\Gamma$, and thus, an upper bound for $\Omega$.
    Thus, by Zorn's lemma, $\operatorname{SubTr}(\Gamma)$ has maximal elements, which are called {\it maximal subtrees}.
	Note that for a maximal subtree $T\subseteq\Gamma$ one has that $T^0=\Gamma^0$ 
	(see \cite[Proposition I.11]{ser:trees}).
\end{rem}

\subsection{Categories and Groupoids}

The algebraic structures that will be studied in the subsequent sections are groupoids.
For this purpose we recall the basic notions in category theory. For further details, the interested the reader is referred to \cite{brown:topgpd}.

A {\it category} $\caC$ consists of
\begin{itemize}
    \item[(i)] a class $\ob(\caC)$, called the class of {\it objects} of $\caC$;
    \item[(ii)] for each $x,y$ in $\ob(\caC)$ a set $\caC(x,y)$ called the set of {\it morphisms} in $\caC$ from $x$ to $y$;
    \item[(iii)] a function, called {\it composition}, which assigns to each $g\in \caC(y, z)$  and to each $f\in\caC(x, y)$ an element $gf\in\caC(x, z)$; that is, composition is a function
$\caC(y, z)\times\caC(x,y)\to\caC(x,z)$.
\end{itemize}
These data must satisfy the axioms:
\begin{itemize}
    \item[(1)] (Associativity) If $h\in\caC(z,w)$, $g\in\caC(y,z)$, $f\in\caC(x,y)$, then $h(gf)=(hg)f$.
    \item[(2)] (Existence of identities) For each $x$ in $\ob(\caC)$ there is an element $1_x$ in $\caC(x,x)$ such that for all $g\in\caC(w,x)$ and  $f\in\caC(x,y)$ one has $1_x g=g$ and $f1_x=f$.
\end{itemize}
For each $x$ in $\ob(\caC)$ the
identity in $\caC(x,x)$ is unique, since if $1_x$, $1_x'$ are both identities in $\caC(x,x)$, then $1_x=1_x 1_x'=1_x'$.

We shall always assume that the various sets $\caC(x,y)$ are disjoint.

\begin{example}
    \begin{itemize}
        \item[(i)] $\Grp$ is the category whose objects are groups and morphisms are group homomorphisms.
        \item[(ii)] $\Set$ is the category whose objects are all sets and whose morphisms $X\to Y$ are simply the functions from $X$ to $Y$, and whose composition is the usual composition of functions. The identity in $Set(X, X)$ is the identity function $1_X$.
        \item[(iii)] If $G$ is group, then $G$ is also a category with one object, i.e., the identity $1$ of $G$, with morphisms $1\to 1$ the elements of $G$, and with composition the multiplication of $G$. Actually, for this construction one needs only that $G$ is a monoid.
        \item[(iv)] $\Graphs$ is the category which has graphs as objects and graph homorphisms as morphisms. 
        \item[(v)] Let $\Gamma$ be an undirected graph. The category {\bf $\PGam$} of paths on $\Gamma$ has the set $\Gamma^0$ as its set of objects, and for any $x, y$ in $\Gamma^0$ the set $\caP\Gamma(x, y)$ is the set of paths in $\Gamma$ from $x$ to $y$. Composition is given by concatenation of paths and the identity $\caP\Gamma(x,x)$ is the path of length $0$, i.e., the vertex $x$.
    \end{itemize}
\end{example}

A category $\caC$ is said to be {\it small} if the objects form a set.
For a small category $\caC$, we denote by $\arr(\caC)$ the set of morphisms of $\caC$.
For a small category $\caC$, one may define an underlying graph $\Gamma_\caC$ given by $\Gamma_\caC^0=\ob(\caC)$ and $\Gamma_\caC^1=\arr(\caC)$.

Let $\caC$ and $\caD$ be two categories. We say that $\caD$ is a {\it subcategory} of $\caC$ if
\begin{itemize}
    \item[(i)] $\ob(\caD)\subseteq\ob(\caC)$;
    \item[(ii)] for each $x,y$ in $\ob(\caD)$ one has $\caD(x,y)\subseteq\caC(x,y)$;
    \item[(iii)] composition of morphisms in $\caD$ is the same as that for $\caC$;
    \item[(iv)] for each $x\in\ob(\caD)$ the identity in $\caD(x,x)$ is the identity in $\caC(x,x)$.
\end{itemize}
The subcategory $\caD$ of $\caC$ is said to be {\it full} if 
$\caD(x,y)=\caC(x,y)$ for all $x,y$ in $\ob(\caD)$; and it is said to be {\it wide} if $\ob(\caD)=\ob(\caC)$.\\

    Let $\caC$ and $\caD$ be two categories.
    A {\it functor} $F\colon\caC\to\caD$ assigns to each object $x$ of $\caC$ an object $F(x)$ of $\caD$ and to each morphism $g\colon x\to y$ a morphism $F(g)\colon F(x)\to F(y)$ in $\caD$ such that the following hold:
    \begin{itemize}
        \item $F(1_x)=1_{F(x)}$ for each $x$ in $\ob(\caC)$;
        \item $F(gh)=F(g)F(h)$ whenever $gh$ is defined.
    \end{itemize}		
Clearly for any category $\caC$ one has an identity functor $1_\caC\colon\caC\to\caC$ which is defined to be the identity map on objects and arrows.

A {\it natural homomorphism} $\alpha\colon F\to G$ of functors $F,G\colon \caC\to\caD$ is a family of morphisms $\alpha_x\colon F(x)\to G(x)$, $x$ in $\ob(\caC)$, such that the diagram commutes for all $\beta\in\caC(x,y)$:
\begin{center}
    \begin{tikzcd}[column sep=normal, row sep=large]
			& F(x) \ar[r, "\alpha_x"] \ar[d, swap, "F(\beta)"]
			& G(x) \ar[d,"G(\beta)"]\\
			& F(y) \ar[r, "\alpha_y"]
			& G(y)
    \end{tikzcd}
\end{center}

A functor $F\colon\caC\to\caD$ is an {\it equivalence} if there exists a functor 
$G\colon\caD\to\caC$ such that $GF$ and $FG$ are naturally isomorphic to the identity functors $1_\caC$ and $1_\caD$, respectively. If such an equivalence exists, we say that the categories $\caC$ and $\caD$ are {\it equivalent}.

\begin{defi}
    A category whose objects form a set and in which every morphism is  an isomorphism is called a {\it groupoid}.
\end{defi}

Equivalently, a groupoid $\caG$ is a small category in which all the morphisms are invertible.
A groupoid can also be viewed as a generalization of a group 
which has a partially-defined product. 
We denote by $\caG^{(0)}$ the set of objects and call it the  {\it unit space} of $\caG$, and by $\caG^{(1)}$ the set of morphisms of $\caG$.
We call {\it units} the elements of $\caG^{(0)}$.
We denote by  $r,s\colon\caG\to\caG^{(0)}$ the {\it range} and {\it source} maps, respectively. 
We call $\caG^{(2)}=\{\,(g_1,g_2)\in\caG\times\caG\mid s(g_1)=r(g_2)\,\}$ the set of {\it composable pairs}.
Note that $\caG^{(0)}=s(\caG)=r(\caG)$ and the units of $\caG$ are the identity morphisms of the category $\caG$, 
in the sense that $gs(g)=g=r(g)g$ for all $g\in\caG$.

A groupoid $\caG$ is said to be {\it connected}
if the map $\caG^{(1)}\to\caG^{(0)}\times\caG^{(0)}$, $g\mapsto(r(g),s(g))$ is surjective.\\

For $x\in\caG^{(0)}$, we say that $\caG_x^x=\{g\in\caG\mid s(g)=x=r(g)\}$ is the {\it isotropy group} of $\caG$ at $x$.
We call $\mathrm{Iso}(\caG)=\bigcup_{x\in\caG^{(0)}} \caG_x^x =\{g\in\caG\mid s(g)=r(g)\}$ the {\it isotropy} of $\caG$.
By defining 
\begin{align}
    \mathrm{Iso}(\caG)^{(0)}&=\caG^{(0)}\notag\\
    \mathrm{Iso}(\caG)(x,y)&=
    \begin{cases}
        \mathrm{Iso}(\caG)(x,y) &\mbox{ if }x=y\\
        \hfil\emptyset\hfil &\mbox{ if }x\ne y\\
    \end{cases}\notag
\end{align} 
for $x,y\in\caG^{(0)}$,
one verifies easily
that the isotropy is a subgroupoid of $\caG$.\\

A {\it homomorphism of groupoids} $f\colon\caG\to\caH$ is essentially a functor, i.e., it consists of a pair of functions 
$f^{(0)}\colon\caG^{(0)}\to \caH^{(0)}$,
$f^{(1)}\colon\caG^{(1)}\to\caH^{(1)}$,  preserving all the structure. 
Thus groupoids and their homomorphisms form a category $\Grpd$.

Let $\caG$ be a groupoid. A {\it subgroupoid} of $\caG$ is a subcategory $\caH$ of $\caG$ such that if $h\in\caH$, then $h^{-1}\in\caH$; that is, $\caH$ is a subcategory which is also a groupoid.
We say $\caH$ is {\it full (wide)} if $\caH$ is a {\it full (wide)} subcategory of $\caG$.

Let $f\colon\caG\to\caH$ be a homomorphism of groupoids.
We define the {\it kernel} and the {\it image} of $f$ to be the subsets
$\ker f=\{\, g\in\caG\mid f(g)\in\caH^{(0)}\,\}$ and 
$\image\, f=\{\,f(g)\mid g\in\caG\,\}$, respectively.
Notice that if $f\colon\caG\to\caH$ is a homomorphism of
groupoids then $\image f$ is not usually a subgroupoid of $\caH$.

The following proposition will be useful for our purpose.
\begin{prop}[\cite{stachura}, Proposition 9]
\label{prop:homoinjker}
    A homomorphism of groupoids $\phi\colon\caG\to\caH$ is injective if and only if $\ker\phi=\caG^{(0)}$.
\end{prop}

\begin{example}
 \begin{itemize}
     \item[(i)] A group, regarded as a category with one object, is also a groupoid.
     \item[(ii)] The category of reduced paths in a tree is a groupoid.
 \end{itemize}
\end{example}

\begin{rem}
\label{rem:forgetfulfunctor}
    An important class of functors are {\it forgetful functors},
    a general term that is used for any functor that forgets structure.
    For example, $U\colon\Grp\to\Set$ sends a group to its underlying set and a group
    homomorphism to its underlying function.
    We will use the forgetful functor $F\colon\Grpd\to\Graphs$ which sends a groupoid $\caG$ to its underlying graph $\Gamma_\caG$ and a groupoid homomorphism $f=(f^{(0)},f^{(1)})$ to its underlying graph homomorphism.
\end{rem}

\section{Groupoid actions}
\label{s:gpdactions}
The notion of a groupoid action on a space is a straightforward generalization of
group actions. 
Much of this section is inspired by \cite{deaconu} and \cite{kummuh}.

\begin{defi}
	\label{def:gpdactionset}
	A groupoid $\caG$ is said to {\it act (on the left)} on a set $X$ 
	if there are given a surjective map $\varphi\colon X\to\caG^{(0)}$,
	called the {\it momentum},
	and a map $\caG\circledast X\to X$, $(g,x)\mapsto gx$,
	where
	$\caG\circledast X=\{(g,x)\in \caG\times X \mid s(g)=\varphi(x) \}$ is called the {\it fibered product} of $\caG$ on $X$,
	that satisfy
	\begin{itemize}
		\item[(A1)] $\varphi(g x)=r(g)$ for all $(g,x)\in \caG\circledast X$;
		\item[(A2)] $(g_1,g_2)\in\caG^{(2)}, (g_2,x)\in\caG\circledast X$ implies
		$(g_1 g_2,x), (g_1,g_2 x)\in\caG\circledast X$ and
		\begin{equation*}
		g_1(g_2 x)=(g_1g_2)x;
		\end{equation*}
		\item[(A3)] $\varphi(x) x =x$ for all $x\in X$.
	\end{itemize}
	We say that $X$ is a {\it left $\caG$-set}.
	A left $\caG$-set $X$ is said to be {\it transitive} if given $x,y\in X$ there exists $g\in\caG$
	such that $g x=y$.	
	The action is said to be {\it free} if $gx=x$ for some $x$ implies $g=\varphi(x)\in\caG^{(0)}$.
	The set of orbits $\caG\ast x=\{ gx\mid g\in\caG, s(g)=\varphi(x) \}$ is denoted by $\caG\backslash X$.
\end{defi}

\begin{rem}
\label{rem:actiongpd}
The fibered product $\caG\circledast X$ has a natural structure of groupoid, called the {\it action groupoid} or {\it semi-direct product} and denoted by $\caG\ltimes X$, where source and range maps are given by 
\begin{align}
    s(g,x)&=(s(g),x)=(\varphi(x),x)\notag\\
    r(g,x)&=(r(g),g\cdot x)=(\varphi(g x),gx),\notag
\end{align}
the set of composable pairs is given by 
\begin{equation*}
    (\caG\ltimes X)^{(2)}=\{\, \big((g_1,x_1),(g_2,x_2)\big)\mid x_1=g_2 x_2\,\},
\end{equation*}
with operations
\begin{align}
    (g_1, g_2x_2)(g_2, x_2)&=(g_1 g_2,x_2),\notag\\
    (g,x)^{-1}&=(g^{-1},g x),\notag
\end{align}
and the unit space $(\caG\ltimes X)^{(0)}$ may be identified with $X$ via the map
\begin{equation*}
    i\colon X\to\caG\ltimes X,\quad i(x)=(\varphi(x),x).
\end{equation*}
\end{rem}

\begin{defi}
Let $\caG$ be a groupoid and let $X$ and $X'$ be left $\caG$-sets.
A {\it morphism} of left $\caG$-sets is a map $F\colon X\to X'$ such that the following diagrams commute
\begin{center}
    \begin{tikzcd}[column sep=normal, row sep=large]
	& X \ar[rr, "F"] \ar[dr, "\varphi"]
	& 
	& X' \ar[dl,"\varphi'"]
	&
	&
	& \caG\ast X\ar[rr] \ar[d, "\iid\times F"]
	&
	& X\ar[d, "F"]
	&[3.0em]\\
	& 
	& \caG^{(0)}
	&
	&
	&
	&\caG\ast X'\ar[rr]
	&
	&X'
    \end{tikzcd}
\end{center}
\end{defi}

We now introduce the notions of groupoid cosets and of conjugation between subgroupoids.

\begin{defi}
\label{def:gpdcosets}
    Let $\caH$ be a wide subgroupoid of a groupoid $\caG$. 
    We define a relation $\sim_\caH$ on $\caG$ by 
    \begin{equation}
	\label{eq:relcosets}
	g_1\sim_\caH g_2\,\Longleftrightarrow\, \mbox{ there exists } h\in\caH :\,g_1=g_2 h.
    \end{equation}
\end{defi} 
It is straightforward to verify that $\sim_\caH$ is an equivalence relation.

\begin{defi}
	The equivalence classes of $\sim_\caH$ are called {\it cosets} and denoted by
	\begin{equation*}
	g\caH=\{\, gh\mid h\in\caH, s(g)=r(h)\,\},
	\end{equation*}
	and the set of cosets is denoted by $\caG \backslash \caH=\{\,g\caH\mid g\in\caG\,\}$.
	As in group theory, one may choose
	a set of coset representatives called a {\it transversal}.
\end{defi}

\begin{prop}
Let $\caH$ be a wide subgroupoid of a groupoid $\caG$. 
For $g_1,g_2\in\caG$, one has that $g_1\in g_2\caH$ if and only if $g_2\in g_1\caH$.
When either condition holds, $g_1\caH=g_2\caH$.
\end{prop}
\begin{proof}
Since $g_1$ and $g_2$ are arbitrary, it sufficies to prove one implication.
Let $g_1\in g_2\caH$. Then there exists $h\in\caH$ such that $g_1=g_2 h$.
Thus, $g_2=g_1 h^{-1}\in g_1\caH$. 

Let $a=g_1 l\in g_1\caH$. Then $a=g_2 h l\in g_2\caH$, since $r(l)=s(g_1)=s(h)$.
Hence, one has that $g_1\caH\subseteq g_2\caH$.
Similarly, one proves that $g_2\caH\subseteq g_1\caH$.
Thus, $g_1\caH=g_2\caH$.
\end{proof}

One has the following straightforward proposition.
\begin{prop}
    Let $\caH$ be a wide subgroupoid of a groupoid $\caG$. 
    Then $\caG\backslash\caH$ is a left $\caG$-set with momentum map $\varsigma\colon\caG\backslash\caH\to\caG^{(0)}$ given by 
    \begin{equation*}
	\varsigma(g\caH)=r(g),
    \end{equation*}
    and the action 
    $\caG\circledast\caG\backslash\caH\to\caG\backslash\caH$ defined by \begin{equation*}
	(g_1,g_2\caH)\mapsto g_1g_2\caH.
    \end{equation*}
\end{prop}
\begin{proof}
One verifies easily that (A1)-(A3) of Definition \ref{def:gpdactionset} are satisfied.
\end{proof}

\begin{rem}
    In \cite{spinosa} the right cosets are characterized by the stabilizer subgroupoid,   as in the classical case of groups. 
    That is, given a groupoid $\caG$ acting on the left on a set $X$, for $x\in X$ one has an isomorphism of left $\caG$-sets $\caG\backslash\Stab_\caG(x)\to\caG \ast x$, 
    $g\Stab_\caG(x)\mapsto gx$.
\end{rem}

The concept of conjugation for subgroupoids of a given groupoid is rather recent and unexplored.
The following definition is taken from \cite{burnsidegpd}.

\begin{defi}
    \label{def:conj}
    Let $\caG$ be a groupoid and let $\caK$, $\caH$ be subgroupoids of $\caG$.
    We say that $\caK$ and $\caH$ are {\it conjugated } (or {\it conjugally equivalent })
    if there are a functor $F\colon \caK \to \caH$ which is an equivalence of categories and 
    a family $\{ g_x\}_{x\in\caK^{(0)}}\subseteq\caG$ such that
    \begin{itemize}
        \item[(i)] $g_x\in \caG(F(x), x)$ for all $x\in\caK^{(0)}$;
        \item[(ii)] for all $k\in\caK(x_2,x_1)$ one has that $F(k)=g_{x_2}\, k\, g_{x_1}^{-1} \in\caH\big(F(x_2),F(x_1)\big)$.
    \end{itemize}
\end{defi}

\begin{rem}
Equivalently, as shown in \cite{burnsidegpd},
two subgroupoids $\caK$ and $\caH$ of $\caG$ with monomorphisms $\tau_\caK\colon\caK\to\caG$ and 
$\tau_\caH\colon\caH\to\caG$ 
are conjugated if there is an equivalence
$F\colon\caK\to\caH$ between their underlying categories and a natural transformation $\mathfrak{g}\colon\tau_\caH\circ F\to\tau_\caK$. 
It follows that $\mathfrak{g}$ is actually a natural isomorphism.
It is also shown in \cite{burnsidegpd} that the conjugacy relation is
reflexive, symmetric and also transitive, i.e., it is an equivalence relation on the set of all subgroupoids of a given groupoid.

\end{rem}
 
In contrast with the classical group setting, conjugated subgroupoids are not
necessarily isomorphic. In fact, we only know that conjugated subgroupoids
have equivalent underlying categories (see \cite{burnsidegpd}, Example 4.9).

We will use conjugated subgroupoids such that the functor $F$ is an injective equivalence of categories. In particular, one has the following lemma.

\begin{lem}
	\label{lem:injconj}
	Let $\caG$ be a groupoid and let $\caK$, $\caH$ be conjugated subgroupoids of $\caG$.
	Let $i_{g}\colon\caK\to\caH$ denote the conjugation map, i.e., 
	$i_g(k)=g_{r(\gamma)}\, k\, g_{s(\gamma)}^{-1}$.
	If the equivalence of categories $F\colon\caK\to\caH$ is injective, then $i_g$ is an injective
	homomorphism of groupoids.
\end{lem}
\begin{proof}
	We first prove that $i_g$ is a groupoid homomorphism.
	For $x\in\caK^{(0)}$, one has that 
	\begin{equation}
	i_g(x)=g_x x g_x^{-1}= g_xg_x^{-1}=F(x)\in\caH^{(0)}.
	\end{equation}
	Moreover, for $k,l\in\caK$ such that $s(k)=r(l)$, one has
	\begin{equation}
	i_g(kl)= g_{r(k)}\,k\,l\,g_{s(l)}^{-1}= g_{r(k)}\,k\, g_{s(k)}^{-1} \,g_{r(l)}\,l\,g_{s(l)}^{-1}= i_g(k)\,i_g(l).
	\end{equation}
	Hence, $i_g$ is a groupoid homomorphism.
	Suppose that there exist $h,h'\in\caH$ such that $i_g(h)=i_g(h')$. Then one has
	\begin{equation}
	g_{r(h)}\,h\,g_{s(h)}^{-1}=g_{r(h')}\,h'\,g_{s(h')}^{-1},
	\end{equation}
	which implies that 
	\begin{equation}
	r(g_{r(h)})=r(g_{r(h')})\quad\mbox{ and }\quad s(g_{s(h)}^{-1})=s(g_{s(h')}^{-1}).
	\end{equation}
	Hence, by (i) of Definition \ref{def:conj}, one has that
	\begin{equation}
	F(r(h))=F(r(h'))\quad\mbox{ and }\quad F(s(h))=F(s(h')).
	\end{equation}
	Since $F$ is injective, this implies that $r(h)=r(h')$ and $s(h)=s(h')$.
	Then one has that
	\begin{equation}
	\begin{aligned}
	h&=g_{r(h)}^{-1}\,g_{r(h)}\, h\, g_{s(h)}^{-1}\, g_{s(h)}\\
	&=g_{r(h)}^{-1}\,g_{r(h')}\, h'\, g_{s(h')}^{-1}\, g_{s(h)}\\
	&=g_{r(h')}^{-1}\,g_{r(h')}\, h'\, g_{s(h')}^{-1}\, g_{s(h')}\\
	&=h'.
	\end{aligned}
	\end{equation}
	That is, $i_g$ is injective.
\end{proof}

\subsection{Groupoids actions on graphs}
\label{ss:gpdactionsgraphs}

\begin{defi}
\label{def:gpdaction}
    Let $\Gamma=(\Gamma^0,\Gamma^1)$ be a graph.
    We say that a groupoid $\caG$ {\it acts on $\Gamma$} if $\caG$ acts on both spaces $\Gamma^0$ and $\Gamma^1$ in a compatible way. 
    This means that there are
    a surjection $\varphi\colon\Gamma^0\to\caG^{(0)}$, called the {\it momentum}, such that 
    \begin{equation}
        \varphi\circ o=\varphi\circ t\colon\Gamma^1\to\caG^{(0)},
    \end{equation}
    and maps
    \begin{equation}
        \mu^0\colon\caG\circledast\Gamma^0\rightarrow\Gamma^0\,\quad\mbox{ and }\quad\,
        \mu^1\colon\caG\circledast\Gamma^1\rightarrow\Gamma^1
    \end{equation}
    which satisfy $\mathrm{(A1)-(A3)}$ and 
    \begin{align}
        o\big(\mu^1(\gamma, e)\big)&=\mu^0(\gamma, o(e)),\\ t\big(\mu^1(\gamma, e)\big)&=\mu^0(\gamma\cdot t(e))
    \end{align}
    for all $(\gamma,e)\in\caG\circledast\Gamma^1$.
    \end{defi}

\begin{rem}
\label{rem:grunion}
Since $\varphi\circ o=\varphi\circ t$, one has that $\Gamma$ is a union of graphs $x\Gamma$ for $x\in\caG^{(0)}$, where $x\Gamma^0=\varphi^{-1}(x)$, and $x\Gamma^1=t^{-1}\big(\varphi^{-1}(x)\big)$.
\end{rem}

For $x\in\Gamma^0$ and $e\in\Gamma^1$, 
we denote by $\caG\ast x=\{\, \mu^0(g,x)\mid g\in\caG, s(g)=\varphi(x)\,\}$ the orbit of $x$
and by $\caG\ast e=\{\, \mu^1(g,e)\mid g\in\caG, s(g)=\varphi(t(e))\,\}$ the orbit of $e$.

\begin{example}
    Let $\Gamma$ be the graph
    \begin{center}
    \tikzset{every loop/.style={min distance=10mm,looseness=10}}
	\begin{tikzpicture}[->,>=stealth', every loop/.style={}]
	   \node[style=circle,inner sep=0pt, minimum size=1.5mm,label] (v) at (-2,0) {$v$};
	   \node[style=circle,inner sep=0pt, minimum size=1.2mm,label] (w) at (2,0) {$w$};

    \path[->] (v) edge [in=120,out=200,loop,distance=2cm] node[auto] {$a_1$} (v);
    \path[->] (v) edge  [in=-10,out=70,loop,distance=2cm] node[auto] {$a_2$} (v);
    \path[->] (v) edge  [in=-45,out=-125,loop,distance=2cm] node[below] {$a_3$} (v);

    \path[->] (w) edge [in=120,out=200,loop,distance=2cm] node[auto] {$b_1$} (w);
    \path[->] (w) edge  [in=-10,out=70,loop,distance=2cm] node[auto] {$b_2$} (w);
    \path[->] (w) edge  [in=-45,out=-125,loop,distance=2cm] node[below] {$b_3$} (w);
	\end{tikzpicture}
    \end{center}
    and let $\caG$ be the connected groupoid with unit space $\caG^{0}=\{v,w\}$ and isotropy
    $S_3$.
    Since $S_3^v$ and $S_3^w$ are isomorphic groups, there is an isomorphism $\beta\colon S_3^w\to S_3^v$, i.e., 
    $\caG$ is the groupoid with the following generators
    \begin{center}
        \begin{tikzpicture}[->,>=stealth', every loop/.style={}]
            \node[style=circle,inner sep=0pt, minimum size=1.5mm,label] (u1) at (-1,0) {$S_3^v$};
            \node[style=circle,inner sep=0pt, minimum size=1.2mm,label] (u2) at (1,0) {$S_3^w$};
            \path[->]
            (u2) edge node[above] {$f$} (u1);
        \end{tikzpicture}
    \end{center}
    and relations $fgf^{-1}=\beta(g)$ for all $g\in S_3^w$.
    Then $\caG$ acts on $\Gamma$ by permutation, i.e., if we suppose that 
    $S_3^w=\mathrm{Sym}(\{1,2,3\})$, then for $g\in S_3^w$ one has $\mu^1(g,b_i)=b_{g(i)}$,
    $\mu^1(\beta(g), a_i)=a_{g(i)}$
    and $\mu^1(f,b_i)=a_i$ for $i=1,2,3$.
    
\end{example}

\begin{example}
    Let $\Gamma$ be the graph
    \begin{center}
        \begin{tikzpicture}[->,>=stealth', every loop/.style={}]
            \node[style=circle,inner sep=0pt, minimum size=1.5mm,label] (v) at (-1.5,0) {$v_1$};
            \node[style=circle,inner sep=0pt, minimum size=1.2mm,label] (w) at (1.5,0) {$v_2$};
            \path[->]
            (v) edge [in=150,out=30,distance=1.2cm] node[above] {$b$} (w)
            (w) edge [in=-30,out=210,distance=1.2cm] node[below] {$c$} (v); 
            \path[->] (v) edge  [in=140,out=220,loop,distance=1.2cm] node[left] {$a_1$} (v);
            \path[->] (v) edge  [in=130,out=230,loop,distance=3cm] node[left] {$a_2$} (v);
            \path[->] (w) edge  [in=40,out=-40,loop,distance=1.2cm] node[right] {$d_1$} (w);
            \path[->] (w) edge  [in=50,out=-50,loop,distance=3cm] node[right] {$d_2$} (w);    
        \end{tikzpicture}
    \end{center}
    and let $\caG$ be the connected groupoid satisfying the following:
    \begin{itemize}
        \item $\caG^{(0)}=\{u_1,u_2\}$;
        \item $g\in\caG$ is such that $s(g)=u_1$, $r(g)=u_2$ and $gf_1g^{-1}=f_2$;
        \item $\caG_{u_i}^{u_i}=\{u_i,f_i\}\cong\Z/2\Z$, for $i=1,2$. That is, $f_i^2=u_i$.
    \end{itemize}
    
    By defining $\varphi(v_i)=u_i$ and 
    \begin{align}
        \mu^0(g,v_1)&= v_2,\notag\\
        \mu^1(g,a_i)&=d_i,\quad i=1,2,\notag\\
        \mu^1(g,c)&=b, \notag\\
        \mu^1(f_1,a_1)&=a_2,\notag\\
        \mu^1(f_2,d_1)&=d_2\notag\\
        \mu^1(f_1,c)&=b,\notag\\
        \mu^1(f_2,b)&=c,\notag
    \end{align}
    we define an action of $\caG$ on $\Gamma$.
\end{example}

\begin{defi}
    Let $\caG$ be a groupoid acting on a graph $\Gamma=(\Gamma^0,\Gamma^1)$. 
    Then we define $Q(\caG,\Gamma)$ by 
    \begin{align}
        Q(\caG,\Gamma)^0=\{\, \caG\ast x \mid x\in\Gamma^0\,\},\\
       Q(\caG,\Gamma)^1=\{\, \caG\ast e \mid e\in\Gamma^1\,\}.
    \end{align}     
    Then one has maps
    \begin{align}
        \label{eq:qo}
	o&\colon Q(\caG,\Gamma)^1\to Q(\caG,\Gamma)^0,
        \quad &&o(\caG\ast e)=\caG\ast o(e),\\
        \label{eq:qt}
	t&\colon Q(\caG,\Gamma)^1\to Q(\caG,\Gamma)^0,
        \quad &&t(\caG\ast e)=\caG\ast t(e),\\
        \label{eq:qbar}
	\,\bar{ }\,&\colon Q(\caG,\Gamma)^1\to Q(\caG,\Gamma)^1,
        \quad && \overline{\caG\ast e}=\caG\ast \bar{e}.
    \end{align}
\end{defi}

\begin{prop}
    The maps $o$, $t$ and $\,\bar{ }\,$ in \eqref{eq:qo}, \eqref{eq:qt} and \eqref{eq:qbar} are well defined.
\end{prop}
\begin{proof}
    Suppose that $\caG\ast e=\caG\ast f$, for $e,f\in\Gamma^1$.
    That is, there exists $g\in\caG$ such that $s(g)=\varphi(o(e))$
    and $\mu^1(g,e)=f$.
    Then one has
    \begin{equation*}
        \mu^0(g,o(e))=o(\mu^1(g,e))=o(f).
    \end{equation*}
    Thus, $o(\caG\ast e)=\caG\ast o(e) =\caG\ast o(f)=o(\caG\ast f)$.
    Similarly, one proves that $t(\caG\ast e)=t(\caG\ast f)$.
    Moreover, since $\mu^1(g,e)=f$ one has that 
    $\mu^0(g,o(\bar{e}))=\mu^0(g,t(e))=t(f)$
    and $\mu^0(g,t(\bar{e}))=\mu^0(g,o(e))=o(f)$.
    Hence $\mu^1(g,\bar{e})=\bar{f}$.
    Thus, one concludes that 
    $\overline{\caG\ast e}=\caG\ast\bar{e}=\caG\ast\bar{f}
    =\overline{\caG\ast f}$.
\end{proof}

\begin{prop}
    For all $e\in\Gamma^1$
    the maps $o$, $t$ and $\,\bar{ }\,$ in \eqref{eq:qo}, \eqref{eq:qt} and \eqref{eq:qbar} satisfy
\begin{itemize}
    \item[(i)] $\overline{\overline{\caG\ast e}}=\caG\ast e$;
    \item[(ii)] $o(\overline{\caG\ast e})=t({\caG\ast e})$ and $t(\overline{\caG\ast e})=o(\caG\ast e)$.
\end{itemize}   
\end{prop}
\begin{proof}
    Let $e\in\Gamma^1$. Then one has
    \begin{itemize}
        \item[(i)] $\overline{\overline{\caG\ast e}}
                =\overline{\caG\ast \bar{e}}
                =\caG\ast \bar{\bar{e}}
                =\caG\ast e$;
        \item[(ii)] $o(\overline{\caG\ast e})
                =o(\caG\ast\bar{e})
                =\caG\ast o(\bar{e})
                =\caG\ast t(e)
                =t({\caG\ast e})$ and\\
                $t(\overline{\caG\ast e})
                =t(\caG\ast\bar{e})
                =\caG\ast t(\bar{e})
                =\caG\ast o(e)
                =o({\caG\ast e})$.
    \end{itemize}
\end{proof}

However, $\overline{\caG\ast e}\ne \caG\ast e$ is not necessarily satisfied, as showed in the following example.
\begin{example}
    Let $\Gamma$ be the graph
    \begin{center}
        \begin{tikzpicture}[->,>=stealth', every loop/.style={}]
            \node[style=circle,inner sep=0pt, minimum size=1.5mm,label] (v) at (0,0) {$v$};
            \path[->] (v) edge  [in=60,out=-30,loop,distance=1cm] node[right] {$e$} (v);
        \end{tikzpicture}
    \end{center}
    and let $\caG$ be such that 
    $\caG^{0}=\{x\}$ and $\caG^{(1)}=\{x,g,g^{-1}\}$.
    Then $\caG$ acts on $\Gamma$ by defining
    $\varphi(v)=x$, $\mu^0(x,v)=v$ and $\mu^1(g,e)=\bar{e}$.
\end{example}

\begin{defi}
    Let $\caG$ be a groupoid acting on a graph $\Gamma$. 
    We say that $\caG$ {\it is acting without inversion of edges} if for all $(\gamma,e)\in\caG\circledast\Gamma^1$ one has 
    $\mu^1(\gamma,e)\ne \bar{e}.$
\end{defi}

Then, for a groupoid $\caG$ acting on a graph $\Gamma$, 
one has that $Q(\caG,\Gamma)$ is a graph if and only if $\caG$ is acting on $\Gamma$ without inversion of edges. 

\begin{defi}
Let $\caG$ be a groupoid acting without inversion on a graph $\Gamma$. We say that $\Gamma$ is a {\it $\caG$-forest} if $x\Gamma$ is a tree for every $x\in\caG^{(0)}$.
\end{defi}


\subsection{Cayley graphs}

\begin{defi}
    Let $\caG$ be a groupoid. A subset $S$ of $\caG$ 
    is said to be an {\it admissible system} for $\caG$ if it does not contain any unit of $\caG$ and 
    if $S=S^{-1}$, where $S^{-1}=\{s^{-1}\mid s\in S\}$.
    We say that an admissible system $S$ is a {\it generating system} if the groupoid generated by $S$ coincides with $\caG$.
\end{defi}

\begin{defi}
    Given a groupoid $\caG$ and a generating system $S$ of $\caG$, the {\it Cayley graph} $\Gamma(\caG,S)$ of $(\caG,S)$ is the graph defined by 
    \begin{align}
	\Gamma(\caG,S)^0&= \caG \tag{i}\\
	\Gamma(\caG, S)^1&=\{\,(gs,g)\mid (g,s)\in \caG\circledast S\,\}\tag{ii}
	\end{align}
	where 
	the maps $t$, $o$ and $\,\bar{ }$ are given by
	\begin{equation}
		t\big((gs,g)\big)= gs,\quad
		o\big((gs,g)\big)=g,\quad
		\overline{(gs,g)}=(g,gs).
	\end{equation}
\end{defi}
Since for any unit $x\in\caG^{(0)}$ one has $x\notin S$, one has that
$\Gamma(\caG,S)^1\subseteq \caG\circledast\caG\setminus\Delta(\caG)$,
where $\Delta(\caG)=\{\,(g,g)\mid g\in\caG\,\}$.

\begin{rem}
    Note that the Cayley graph $\Gamma(\caG,S)$ is fibered on $\caG^{(0)}$ via the map $\varphi=r\colon\Gamma(\caG,\caS)^0=\caG\to\caG^{(0)}$.
    For $x$ in $\caG^{(0)}$ we denote by $\Gamma(\caG, S)_x$ the fiber of $x$, i.e.,
	\begin{align}
	   \Gamma(\caG, S)_x^0&=\varphi^{-1}(x),\notag\\ 
	   \Gamma(\caG, S)_x^1&=o^{-1}(\varphi^{-1}(x)).\notag
	\end{align}
\end{rem}

\begin{prop}
\label{prop:cayley}
Let $\caG$ be a groupoid and $S\subseteq\caG$ be an admissible system.
Then the following are equivalent
    \begin{itemize}
	\item[(i)] $S$ is a generating system;
	\item[(ii)] $\Gamma(\caG,S)_x$ is connected for any $x\in\caG^{(0)}$.
    \end{itemize}
\end{prop}
\begin{proof}
    Fix $x\in\caG^{0}$
    Suppose that (i) holds and let $g,h\in x\caG=\varphi^{-1}(x)$ such that $g\ne h$.
    Since $S$ generates $\caG$, 
    there exist elements $a_1,\dots,a_n\in S$ such that 
    $s(a_i)=r(a_{i+1})$, $s(h)=r(a_1)$, $s(g)=s(a_n)$ and
    $h^{-1}g=a_1\cdots a_n$.
    Hence one has $g=ha_1\cdots a_n$, and thus
    \begin{equation*}
	(h, h a_1)\, (h a_1, ha_1 a_2)\cdots (ha_1 a_2\cdots a_{n-1}, ha_1 a_2\cdots a_{n})
    \end{equation*} 
    is a path from $h$ to $g$. Thus, $\Gamma(\caG,S)_x$ is connected.
	
    Now suppose that (ii) holds and let $g\in\caG$ with $r(g)=x$.
    Then there exists a path from $x$ to $g$ given by
    \begin{equation*}
	(x, b_1)\, (b_1, b_2)\cdots ( b_{m-1}, b_{m})
    \end{equation*} 
    where $r(b_i)=x$ for all $i=1,\dots,m$ and $b_m=g$. 
    By definition, $b_1\in S$ and $b_{j+1}^{-1}b_j\in S$.
    So, by induction, $g=b_m\in\langle S\rangle$. 
    Thus $S$ generates $\caG$.
\end{proof}

\section{Graphs of groupoids}
\label{s:graphgpd}
\begin{defi}
    A {\it graph of groupoids} $\caG(\Gamma)$ based on a connectd graph $\Gamma=(\Gamma^0,\Gamma^1)$ is given by the following data:
    \begin{itemize}
	\itemsep0em
	\item[(i)] a {\it vertex} groupoid $\caG_v$ for every vertex $v\in\Gamma^0$;
	\item[(ii)]an {\it edge} groupoid $\caG_e$ for every edge $e\in\Gamma^1$
		satisfying $\caG_e=\caG_{\bar{e}}$;
	\item[(iii)] an injective homomorphism of groupoids 
    $\alpha_e\colon \caG_e\to\caG_{t(e)}$ for every $e\in\Gamma^{1}$.
    \end{itemize}
If $\Gamma$ is a tree, we call $\caG(\Gamma)$ a {\it tree of groupoids}.
\end{defi}

\begin{hypo}
\label{hy:standingass}
We will asssume that $\alpha_e(\caG_e)$ is a {\it wide} subgroupoid
of $\caG_{t(e)}$ for all $e\in\Gamma^1$, i.e., $\alpha_e(\caG_e)$ has the same unit space as $\caG_{t(e)}$. 
\end{hypo}

\begin{rem}
    \label{rem:spn_category}
    For each $e\in\Gamma^1$ one has a pair of groupoid homomorphisms $(\alpha_e, \alpha_{\bar{e}})$ as follows
    \begin{center}
    \begin{tikzcd}[column sep=normal, row sep=large]
			& \caG_{t(e)} 
			& \caG_e \ar[l,"\alpha_e",labels=above] \ar[r, "\alpha_{\bar{e}}"] 
                & \caG_{o(e)}
    \end{tikzcd}
\end{center}
    Such a diagram is sometimes also called a {\it span of groupoids} in the category $\Grpd$.
\end{rem}
\begin{example}
    Let $\caG(\Gamma)$ be the graph of groupoids	\begin{center}
        \begin{tikzpicture}[->,>=stealth']
        \node[style=circle,fill=red,inner sep=0pt, minimum size=2mm,label=above:$\caG_v$] (n1) at (3,0) {};
        \node[style=circle,fill=blue,inner sep=0pt, minimum size=2mm,label=above:$\caG_w$] (n2) at (1,0)  {};	
        \path[semithick]
        (n1) edge node [left][above] {$\caG_e$} (n2);	
        \path[dashed]
        (n2) edge [bend right] node [right][below] {} (n1);
	\end{tikzpicture}
    \end{center}
    where $\Gamma=\big(\{v,w\}, \{e\}\big)$ and the vertex and edge groupoids are given by 
    \begin{align}
        \caG_v&=\{x_1,x_2,a_v, a_v^{-1}\}, \quad s(a_v)=x_1, \, r(a_v)=x_2,\notag\\
	\caG_w&=\{y_1,y_2,a_w,a_w^{-1}\},\quad
        s(a_w)=y_1, \, r(a_w)=y_2,\notag\\
	\caG_e&=\caG_{\bar{e}}=\{z_1,z_2\}.\notag
    \end{align}
	That is, $\caG(\Gamma)$ is the graph of groupoids
	\vskip0.5cm
	\begin{center}
		\begin{tikzpicture}[->,>=stealth']
		\node[style=circle,fill,inner sep=0pt, minimum size=1.2mm,label=below:$v$] (nv) at (4.5,0) {};
		\node[style=circle,fill,inner sep=0pt, minimum size=1.2mm,label=below:$w$] (nw) at (2.5,0) {};
		\node[style=circle,fill=blue,inner sep=0pt, minimum size=1.2mm,label=below:$y_1$] (n1) at (1,-1) {};
		\node[style=circle,fill=blue,inner sep=0pt, minimum size=1.2mm,label=above:$y_2$] (n2) at (1,1)  {};	
		\node[style=circle,fill=red,inner sep=0pt, minimum size=1.2mm,label=below:$x_1$] (na) at (6,-1) {};
		\node[style=circle,fill=red,inner sep=0pt, minimum size=1.2mm,label=above:$x_2$] (nb) at (6,1)  {};
		\node[style=circle,fill=aogreen,inner sep=0pt, minimum size=1.2mm,label=above:$z_1$] (z1) at (4.2,1.5) {};
		\node[style=circle,fill=aogreen,inner sep=0pt, minimum size=1.2mm,label=above:$z_2$] (z2) at (2.8,1.5)  {};
		
		\path[semithick]
		(nv) edge node [above] {$e$} (nw);
		\path[dashed]
		(nw) edge [bend right] node [right][below] {$\bar{e}$} (nv);
		
		\path[blue,semithick]
		(n1) edge [bend left] node [left] {$a_w$} (n2)
		(n2) edge [bend left] node [right] {$a_w^{-1}$} (n1);
		
		\path[red,semithick]
		(na) edge [bend left] node [left] {$a_v$} (nb)
		(nb) edge [bend left] node [right] {$a_v^{-1}$} (na);
		
		\end{tikzpicture}
	\end{center} 
    The monomorphisms $\alpha_e\colon\caG_e\to\caG_w$ 
    and $\alpha_{\bar{e}}\colon\caG_e\to\caG_v$ are given by 
    \begin{equation*}
        \alpha_e(z_i)=y_i,
	\quad
	\alpha_{\bar{e}}(z_i)=x_i,
	\notag
    \end{equation*} 
    for $i=1,2$. Then one has
    $\image(\alpha_e)=\caG_w^{(0)}$,
    $\image(\alpha_{\bar{e}})=\caG_v^{(0)}$.
 \end{example}

\begin{notation}
\label{n:partiso}
    For $e\in\Gamma^1$, we put $\caH_e=\alpha_e(\caG_e)$.
    We denote by 
    \begin{equation*}
    \phi_e\colon\caH_{\bar{e}}\subseteq\caG_{o(e)}\to
    \caH_e\subseteq\caG_{t(e)}
    \end{equation*} 
    the map given by
    \begin{equation*}
        \phi_e(g)=(\alpha_e\circ \alpha_{\bar{e}}^{-1} )(g),\quad g\in\caH_{\bar{e}}.
    \end{equation*}
    Hence, $\phi_e$ is an isomorphisms between the groupoids $\caH_{\bar{e}}$ and $\caH_e$ and one has $\phi_{\bar{e}}=\phi_e^{-1}$.
\end{notation}

 \begin{defi}
    A {\it graph of partial isomorphisms} $\Gamma^{\mathrm{pi}}$ is 
    given by a graph $\Gamma=(\Gamma^0,\Gamma^1)$ together with
    \begin{itemize}
    \itemsep0em
        \item[(i)] a set $\Gamma_v$ for each vertex $v\in\Gamma^0$;
        \item[(ii)] a partial isomorphism $\phi_e^0\colon\Gamma_{o(e)}\rightharpoonup\Gamma_{t(e)}$ for each edge $e\in\Gamma^1$, 
        \begin{center}
		\begin{tikzcd}[column sep=normal, row sep=large]
			& \dom(\phi_e^0) \ar[r, "\phi_e"] \ar[d, hook]
			& \image(\phi_e^0) \ar[d, hook]\\
			& \Gamma_{o(e)} 
			&\Gamma_{t(e)}
		\end{tikzcd}
	\end{center}
        i.e., a bijection such that 
        the domain and image of $\phi_e^0$ are subsets $\dom(\phi_e^0)\subseteq\Gamma_{o(e)}$ 
	and $\image(\phi_e^0)\subseteq\Gamma_{t(e)}$
        respectively, and such that 
        $\phi_{\bar{e}}^0=(\phi_e^0)^{-1}$.
    \end{itemize}
    If $\Gamma$ is a tree, we call $\Gamma^{\mathrm{pi}}$ a {\it tree of partial isomorphisms}.
\end{defi}

Thus, Notation \ref{n:partiso} introduces a graph of partial isomorphisms.
\begin{rem}
    For a set $X$, the category of partial bijections $\PB(X)$ denotes the category
    whose objects are the subsets of $X$ and whose morphisms are bijections between such objects.
\end{rem}

\begin{rem}
	\label{rem:pipaths}
	Let $\Gamma^{\mathrm{pi}}$ be a tree of partial isomorphisms.
	
	(a) 
	Since $\Gamma$ is a tree, for $v,w\in\Gamma^0$ there exists a unique reduced path 
	$p\in\mathcal{P}_{v,w}$, i.e., $p=e_1\cdots e_n$, $n\in\N$, $o(p)=o(e_n)=w$ and $t(p)=t(e_1)=v$. 
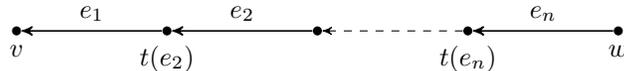
	\begin{center}
		\begin{tikzpicture}[->,>=stealth']
		\node[style=circle,fill,inner sep=0pt, minimum size=1.2mm,label=below:{$v$}] (t1) at (-4,0) {};
		\node[style=circle,fill,inner sep=0pt, minimum size=1.2mm,label=below:{$t(e_2)$}] (t2) at (-2,0) {};
		\node[style=circle,fill,inner sep=0pt, minimum size=1.2mm] (t3) at (0,0) {};
		\node[style=circle,fill,inner sep=0pt, minimum size=1.2mm,label=below:{$t(e_{n})$}] (tn) at (2,0) {};
		\node[style=circle,fill,inner sep=0pt, minimum size=1.2mm,label=below:{$w$}] (on) at (4,0) {};
		
		\path[semithick]
		(t2) edge node [above] {$e_1$} (t1)
		(t3) edge node [above] {$e_2$} (t2)
		(on) edge node [above] {$e_n$} (tn);
		
		\path[dashed]
		(tn) edge node [above] {} (t3);
		
		
		\end{tikzpicture}
	\end{center}
	Each edge $e_i$, $i=1,\dots,n$, carries a partial isomorphism
	\begin{equation*}
	\phi_{e_i}^0\colon D_i\to C_i,
	\end{equation*}
	where $D_i=\dom(\phi_{e_i}^0)\subseteq\Gamma_{o(e_i)}$ and
	$C_i=\image\phi_{e_i}^0\subseteq\Gamma_{t(e_i)}$.
	Let $A_n=\dom(\phi_{e_n}^0)$ and $A_i=D_i\cap C_{i+1}$ for $i=1,\dots,n-1$.
	Suppose that $A_i\ne\emptyset$ for all $i=1,\dots,n$.
	Then 
	\begin{equation}
	\Phi_p:=\phi_{e_1}^0\RestrTo{A_1}\,\circ\,\phi_{e_2}^0\RestrTo{A_2}\,\circ\,\cdots\,\circ\,\phi_{e_n}^0\RestrTo{A_n}
	\end{equation}
	defines a partial isomorphism between $\Gamma_{w}$ and $\Gamma_{v}$.
	
	\vspace{0.2cm}
	(b) Let $U=\bigsqcup_{v\in\Gamma^0} \Gamma_v$ be the disjoint union of the vertex graphs of $\Gamma^{\mathrm{pi}}$. 
	Then the partial isomorphisms $\phi_e^0$, $e\in\Gamma^1$, associated to the edges of $\Gamma$ induce an equivalence relation on $U$ given by
	\begin{equation}
	\label{eq:relpi} 
	x\sim y \,\Longleftrightarrow \mbox{ there exists }p\in\mathcal{P}_{v,w} \mbox{ such that } \Phi_p(y)=x,\,
	x\in\Gamma_{v},\, y\in\Gamma_w.
	\end{equation}
	Clearly, $\sim$ is reflexive. 
	It is also symmetric, since $\Phi_p(x)=y$ implies that $\Phi_{\bar{p}}(y)=x$,
	where $\bar{p}$ is the reversal path of $p$.
	Finally, it is transitive, since $\Phi_p(x)=y$ and $\Phi_q(y)=z$ implies that $o(p)=t(q)$ and
	$\Phi_{pq}(x)=z$.
	
	Therefore one has that for each $x\in \image(\Phi_p)\subseteq \Gamma_{v}$ 
	there exists a unique $y\in\dom(\Phi_p)\subseteq\Gamma_{w}$ 
	such that $x\sim y$.
\end{rem}

\begin{defi}
    Let $A$ be a set, and let $\mathcal{R}$ be an equivalence relation on $A$. 	
    A subset $B$ of $A$ is said to be {\it saturated} with respect to $\mathcal{R}$ if for all $x,y \in A$, $x\in B$ and $x\mathcal{R}y$ imply $y\in B$.  
    For a subset $C$ of $A$, the {\it saturation} of $C$ with respect to $\mathcal{R}$ is the least saturated subset $S(C)$ of $A$ that contains $C$.
\end{defi}

Equivalently, $B$ is saturated if it is the union of a family of equivalence classes with respect to $\mathcal{R}$.

\begin{defi}
    A tree of partial isomorphisms $\Gamma^{\mathrm{pi}}$ is said to be {\it rooted}
    if there exists a vertex $v\in\Gamma^0$, called the {\it root}, such that 
    $\Gamma_v$ is a system of representatives for the relation $\sim$ defined in \eqref{eq:relpi}, i.e., 
    it contains exactly one element for each equivalence class of $\sim$.
\end{defi}

\begin{rem}
\label{rem:treepigraphgpd}
    One may associate to any graph of groupoids $\caG(\Lambda)$ a tree of partial isomorphisms $\Gamma^\mathrm{pi}$ as follows.
    The tree $\Gamma\subseteq\Lambda$ is a maximal subtree of $\Lambda$.
    The sets $\Gamma_v$, $v\in\Gamma^0$, associated to the vertices of $\Gamma$ are given by $\Gamma_v=\caG_v^{(0)}$ and the partial isomorphisms 
    $\phi_e^0\colon \caG_{o(e)}^{(0)}\to \caG_{t(e)}^{(0)}$, $e\in\Gamma^1$, are given by 
    \begin{equation*}
	\phi_e^0=\phi_e\RestrTo{\caG_{o(e)}^{(0)}},
    \end{equation*}
    where $\phi_e\RestrTo{\caG_{o(e)}^{(0)}}$ is the restriction of the map $\phi_e$ defined in Notation \ref{n:partiso} to the objects of $\caG_{o(e)}$.
    Since we will only consider graphs of groupoids where the edge groupoids are wide subgroupoids of the adjacent vertex groupoids, i.e., $\alpha_e(\caG_e)=\caH_e\subseteq\caG_e$ is a wide subgroupoid of $\caG_{t(e)}$ for all $e\in\Gamma^1$,
    we shall only deal with tree of partial isomorphisms $\Gamma^{\mathrm{pi}}$ such that each $\phi_e^0\colon\caG_{o(e)}^{(0)}\to\caG_{t(e)}^{(0)}$, $e\in\Gamma^1$, is bijective. 
    Therefore, each $\Gamma_v$ is a system of representative for the equivalence relation $\sim$ defined in \eqref{eq:relpi}.
\end{rem}

Thus, under our assumptions, we will deal with rooted trees of total isomorphisms.

\begin{rem}
\label{rem:fpi}
    Let $\Gamma^{\mathrm{pi}}$ be a tree of partial isomorphisms.
    Then $\Gamma^{\mathrm{pi}}$ defines naturally a forest $\caF^{\mathrm{pi}}(\Gamma)$ given as follows.
    \begin{align}
        {\caF^{\mathrm{pi}}}(\Gamma)^0&=\bigsqcup_{v\in\Gamma^0}\Gamma_v\notag\\
	{\caF^{\mathrm{pi}}}(\Gamma)^1&=\{\, (e,x,y) \in
        \Gamma^1\times{\caF^{\mathrm{pi}}}(\Gamma)^0\times {\caF^{\mathrm{pi}}}(\Gamma)^0\mid x\in\Gamma_{o(e)}, y\in\Gamma_{t(e)}, \phi_e^0(x)=y  \,\}.\notag
    \end{align}
    Clearly, $\caF^{\mathrm{pi}}(\Gamma)$ is a graph with terminus and origin maps given by the projection on the second and third component, respectively, and inversion given by $(e,x,y)^{-1}=(\bar{e},y,x)$.
    Moreover, it is a disjoint union of graphs since each connected component is given by an equivalence class of the relation $\sim$ defined in \eqref{eq:relpi}.
    Finally, since $\Gamma$ is a tree, for each pair of vertices $x\ne y$, $x$ and $y$ in the same connected component of  $\caF^{\mathrm{pi}}(\Gamma)$ there is at most one path from $x$ to $y$.
    Thus, each connected component of $\caF^{\mathrm{pi}}(\Gamma)$ is a tree and hence $\caF^{\mathrm{pi}}(\Gamma)$ is a forest.
\end{rem}

\begin{defi}
\label{defi:reprpi}
    Let $\Gamma^{\mathrm{pi}}$ be a rooted tree of partial isomorphisms and let $F$ be a forest.
    A {\it representation} of $\Gamma^{\mathrm{pi}}$ on $F$ is a graph homomorphism 
    $\chi\colon \caF^{\mathrm{pi}}(\Gamma)\to F$ from the forest $\caF^{\mathrm{pi}}(\Gamma)$ defined by 
    $\Gamma^{\mathrm{pi}}$ to $F$.
    We say that 
    $F'=\chi(\caF^{\mathrm{pi}}(\Gamma))\subseteq F$ {\it is the representation} of $\Gamma^\mathrm{pi}$ on $F$.
\end{defi}

Given a tree of partial isomorphisms $\Gamma^\mathrm{pi}$, 
we can associate to each edge $e\in\Gamma^1$ 
a set $\Gamma_e$ with injections 
$\alpha_{\bar{e}}\colon\Gamma_e\to\Gamma_{o(e)}$ and 
$\alpha_e\colon\Gamma_e\to\Gamma_{t(e)}$ in such a way that 
$\alpha_{\bar{e}}(\Gamma_e)=\mathrm{dom}(\phi_e)$ and 
$\alpha_e(\Gamma_e)=\image(\phi_e)$.

\begin{defi}
\label{defi:chi}
    Let $F'=\chi(\caF^{\mathrm{pi}}(\Gamma))\subseteq F$ be a representation of a tree of partial isomorphisms 
    $\Gamma^{\mathrm{pi}}$ on $F$.\\
    For $v\in\Gamma^0$, we denote by $\chi(\Gamma_{v})$ 
    the set 
    \begin{equation*}
	\chi(\Gamma_{v})=\{\,\chi(x)\mid 
	x\in \Gamma_v\subseteq\caF^\mathrm{pi}(\Gamma)^{0} \,\}\subseteq {F'}^0
    \end{equation*}
    and we call it the {\it representation of $\Gamma_v$ on $F$}.
    For $e\in\Gamma^1$, we denote by
    \begin{equation*}
	\chi(\Gamma_e)=\{\,\eps\in {F'}^1\mid
        o(\eps)\in\chi(\Gamma_{o(e)}),\, t(\eps)\in\chi(\Gamma_{t(e)})\,\}
	\subseteq{F'}^1
    \end{equation*} 
    and we call it the {\it representation of $\Gamma_e$ on $F$}.
\end{defi}

\begin{defi}
    Let $\caG$ be a groupoid acting on a forest $F$ 
    and let $F'\subseteq F$ be a representation of a tree of partial isomorphisms 
    $\Gamma^{\mathrm{pi}}$ on $F$.
    We say that $F'$ is {\it $\caG$-invariant} if
    $\mu^0(\caG,F'^0)\subseteq F'^0$ and $\mu^1(\caG,F'^1)\subseteq F'^1$.
\end{defi}

\begin{rem}
    Let $\caG$ be a groupoid acting on a forest $F$ with momentum map $\varphi\colon F^0\to\caG^{(0)}$
    and let $F'\subseteq F$ be any $\caG$-invariant representation of a tree of partial isomorphisms 
    $\Gamma^{\mathrm{pi}}$ on $F$.
    The action of $\caG$ on $F$ induces an action of $\caG$ on $F'$ which gives rise to
    a tree of groupoids $\caG_\chi(\Gamma)$ based on $\Gamma$ as follows.
    We put
    \begin{align}
        \caG_{\chi,v}^{(0)}&:=\Gamma_v,\quad v\in\Gamma^0,\notag\\
	\caG_{\chi,e}^{(0)}&:=\Gamma_e, \quad e
        \in\Gamma^1.\notag
    \end{align}
    Then we define $\caG_{\chi,v}$ to be the groupoid on $\caG_{\chi,v}^{(0)}$ whose morphisms are given by the action of $\caG$ on $\chi(\Gamma_v)\subseteq F^0$, i.e.,
    \begin{equation}
    \label{eq:chiv}
        \caG_{\chi,v}=\{\, g\in\caG\mid s(g)=\varphi(\nu),\, \nu\in\chi(\Gamma_v) \mbox{ and } 
	\mu(g,\nu)\in\chi(\Gamma_v) \,\}.
    \end{equation}
    Similarly, we define $\caG_{\chi,e}$ to be the groupoid on $\caG_{\chi,e}^{(0)}$
    whose morphisms are given by the action of $\caG$ on $\chi(\Gamma_e)\subseteq F^1$, i.e.,
    \begin{equation}
    \label{eq:chie}
        \caG_{\chi,e}=\{\, g\in\caG\mid s(g)=\varphi(o(\eps)),\, \eps\in\chi(\Gamma_e) \mbox{ and }  \mu(g,\eps)\in\chi(\Gamma_e) \,\}.
    \end{equation}
\end{rem}

\begin{defi}
    Let $\caG(T)$ be a tree of groupoids and let $F$ be a forest.
    A {\it representation of $\caG(T)$ on $F$} is a representation $\chi$ 
    of the tree of partial isomorphisms $T^{\mathrm{pi}}$ underlying $\caG(T)$ on $F$
    (see Remark \ref{rem:treepigraphgpd})
    such that there are groupoid isomorphisms 
    between the vertex and edge groupoids of $\caG(T)$ and the groupoids
    defined by the action of $\caG$ on the representations of the $T_v$'s and $T_e$'s
    on $F$, i.e., groupoid isomorphisms
    \begin{equation*}
	\caG_v\to \caG_{\chi,v}, \quad \caG_e\to \caG_{\chi,e},
    \end{equation*}
    for $v\in T^0$ and $e\in T^1$.
\end{defi}

\begin{rem}
    Let $\caG$ be a groupoid acting on a forest $F$.
    The action of $\caG$ on $F$ induces an equivalence relation on $F^0$ and $F^1$,
    which we denote by $\caR_\caG$, defined by
    \begin{equation}
        \begin{aligned}
            v&\caR_\caG w \Longleftrightarrow \mbox{ there exists } g\in\caG: \mu(g,v)=w\\
            e&\caR_\caG f \Longleftrightarrow \mbox{ there exists } h\in\caG: \mu(h,e)=f.
	\end{aligned}
    \end{equation}
\end{rem}

\begin{defi}
\label{def:treeofrepr}
    Let $\caG$ be a groupoid acting on a forest $F$.
    A {\it tree of representatives} $(\caG(T),\chi)$ 
    of the action of $\caG$ on $F$
    is given by a rooted tree of groupoids $\caG(T)$, based on a rooted tree $T$, 
    with vertex groupoids $\caG_x$, 
    $x\in T^{0}$, and a representation $\chi$ of $\caG(T)$ on $F$ 
    such that the saturations of the $\chi(\caG_x^{(0)})$'s, $x\in T^{0}$, 
    with respect to the equivalence relation $\caR_\caG$ 
    give a partition of $F^{0}$.
\end{defi}

\section{The graph of groupoids associated to a groupoid action on a forest}
\label{s:gractforest}

\subsection{Fiber spaces}
    A set $F$ is said to be a {\it fiber space} on a set $X$ if there exists 
    a surjective map $\pi\colon F\to X$, called the {\it projection}, such that $\pi^{-1}(x)$
    is countable for all $x\in X$.
    We call $F_x=\pi^{-1}(x)$ the {\it fiber} of $x$, $x\in X$.
    A {\it section} of $F$ is a map $\sigma\colon X\to F$ such that $\pi\circ\sigma=\iid_X$.
    If $A\subseteq X$ and $\sigma_A\colon A\to F$ is such that $\pi\circ\sigma_A=\iid_A$, 
    we say that $\sigma_A$ is a {\it partial section} of $F$.

\begin{rem}
\label{rem:fiberenumeration}
    In the case we are interested in, the target set $X$ of the projection $\pi$ is countable, and thus the fiber space $F$ is countable too.
    If $F$ is a fiber space on $X$, then there exists a countable family of partial sections of $F$ such that their images form a partition of $F$.
    Moreover, there exists an enumeration of the fibers of $F$, i.e., there exists a map 
    $N\colon F\to\N$ such that $N|_{F_x}\colon F_x\to\N$ is injective for all $x\in X$.
    One may suppose that the enumeration in each fiber of $F$ starts from $1$ 
    and follows the natural enumeration of $\N$.
\end{rem}


\subsection{Groupoid actions on forests}

We will deal with groupoid actions on fiber spaces.

\begin{defi}
\label{def:stab}
    Let $\caG$ be a groupoid acting on a fiber space $(F,\pi)$ on $\caG^{(0)}$.
    If $\sigma\colon A\subseteq\caG^{(0)} \to F$ is a partial section of $F$, 
    we call the {\it stabilizer} of $\sigma$ the subgroupoid
    \begin{equation}
        \Stab_\caG(\sigma)=\{\,g\in \caG\mid g\cdot\sigma(s(g))=\sigma(r(g)), s(g),r(g)\in A\,\}
	\subseteq \caG.
    \end{equation}
\end{defi}

\begin{rem}
    Let $\caG$ be a groupoid acting on a graph $\Gamma=(\Gamma^0,\Gamma^1)$
    (cf. Definition \ref{def:gpdaction}).
    Then $\Gamma^0$ is a fiber space on $\caG^{(0)}$ via the momentum map $\varphi\colon\Gamma^0\to\caG^{(0)}$. 
\end{rem}

\begin{hypo}[Arboretum hypothesis]
\label{hy:arboretum}
    Given a groupoid action on a forest, we will suppose that each fiber of the forest is a tree  
    (cf. \cite{alv08}).
\end{hypo}

\begin{thm}
    \label{cl:parsec}
    Let $\caG$ be a groupoid acting without inversion of edges on a  forest $F=(F^0,F^1)$.
    Then there exists a countable or finite family $\Sigma=\{\sigma_{i}\}_{i\in I}$, $I\subseteq\N$, of partial sections of $F^0$
    with domains $X_{i}\subseteq\caG^{(0)}$, i.e., 
    $\sigma_{i}\colon X_{i}\to F^0$, $i\in I$, such that
    \begin{itemize}
	\item[(i)] $S(U_{i})\cap S(U_{j})=\emptyset$
            for all $i\ne j$, where $U_{i}=\sigma_{i}(X_{i})$ and
		$S(U_{i})$ denotes the saturation of
            $U_{i}$ with respect to $\caR_{\caG}$;
        \item[(ii)] for $V=\bigcup_{i\in I} U_{i}$, one has that
		$V \cap\varphi_0^{-1}(x)$ is connected for all
            $x\in \caG^{(0)}$;
	\item[(iii)] $S(V)= F^0$.
	\end{itemize}
\end{thm} 

	\begin{proof}	
		We construct recursively a sequence of partial sections 
		$\{\sigma_i\}_{i\in I}$, $I \subseteq\N$,
		satisfying (i) and (ii).
		Let $X=\caG^{(0)}$ be the unit space of the groupoid $\caG$.
		The graph $F=(F^0,F^1)$ is fibered on $X$ 
		via the map $\varphi=(\varphi_0,\varphi_1)$, where $\varphi_0\colon F^0\to X$
		and  $\varphi_1=\varphi_0\circ t \colon F^1\to X$.
		By Remark \ref{rem:fiberenumeration}, there exists an enumeration of every fiber of $F$. That is, there exists a bijection $\alpha_{x}\colon \N\to\varphi_0^{-1}(x)$, $x\in X$.
		Let $\sigma_1\colon X_1\subset X\to F^0$ be the partial section whose image is given by the elements corresponding to the smallest number in each fiber. That is, $\image(\sigma_1)=\{\,\alpha_x(1)\mid x\in X\,\}$.
		Let $U_1=\image(\sigma_1)$ and let $S(U_1)$ be the saturation of $U_1$
		with respect to $\caR_{\caG}$.
		Put $C_1=F^0\setminus S(U_1)$.
		If $C_1=\emptyset$, then  $U_1$ is a complete domain for the action of $\caG$ on $F^0$ and the claim follows by choosing $n=1$ and $\Sigma=\{\sigma_1\}$.	
		Otherwise, we use the same argument as above to construct partial sections $\sigma_k$,  $k\ge2$, 
		recursively.\\
		Let $n\ge 1$.
		Suppose that we have constructed $n$ partial sections $\{\sigma_i\}_{1\le i\le n}$
		satisfying conditions (i) and (ii).
		Let $V_{n}=\sqcup_{i=1}^{n} U_i$
		and let $C_{n}=F^0\setminus S(V_{n})$.
		Then either $C_{n}=\emptyset$ and thus (iii) holds for $I=\{1,\dots,n\}$,
		or $C_{n}\ne\emptyset$ and we construct a partial section
		$\sigma_{n+1}$ as follows.
  Since the fibers of $F$ are connected by Hypothesis \ref{hy:arboretum} and the action of $\caG$ on $F$ preserves the distance, there exist elements in $C_n$ which have unitary distance from $S(V_n)$.
		Let $C_{n}'=\{\,x\in C_n\mid \dist(x,V_n)=1\,\}$ and put
		$X_{n+1}=\{\,x\in X\mid \varphi_0^{-1}(x)\cap C_n'\ne\emptyset \,\}$.
		We define $\sigma_{n+1}\colon X_{n+1}\to C_n'$ by
		$\sigma_{n+1}(x)=\min \{\, \varphi_0^{-1}(x)\cap C_n' \,\}$,
		where the minimum is taken over the enumeration of 
		the elements of each fiber $\varphi_0^{-1}(x)$.
		Suppose that $I=\N$ and $F^0\ne\tilde{V}$, where
		$V=\cup_{k=1}^\infty U_k$.
		Then the set $C=\{\,x\in F^0\setminus\tilde{V}\mid\dist(x,V)=1\,\}$
		is nonempty.
		Let $a\in C$. 
		Then there exists $v\in \tilde{U}_k$, for some $k\in\N$, 
		such that $\dist(a,v)=1$.
		Hence, $\dist(a, \tilde{V}_{l})=1$ for all $l\ge k$,
		and one has that if $a\in\varphi^{-1}(\varphi(a))\cap C_l$
		is the $m$-th element in the enumeration, $m\in\N$,
		then
		\begin{equation}
		a\in\bigcup_{1\leq j\leq m} \image(\sigma_{k+j})\subseteq V,
		\end{equation}
		a contradiction.
		Hence, one has that $V=F^0$.
	\end{proof}

\begin{rem}
	\label{rem:mn}
	By the construction of the $\sigma_i$'s used above, one has that $\image(\sigma_{i})$ contains the smallest elements which have distance $1$
	from $V_{i-1}=\bigcup_{k=1}^{i-1} \image(\sigma_k)$ 
	in each fiber $\varphi_0^{-1}(x)$, $x\in\caG^{(0)}$.
	Then for each $i\in\{1,\dots,n\}$ one has a map
	\begin{equation}
	\begin{aligned}
	m_i\colon F^0 &\to \{1,\dots,n\}\\
	a&\mapsto m_i(a),
	\end{aligned}
	\end{equation} 
	where $m_i(a)$ denotes the integer $m_i(a)< i$ such that the vertex $a$ is adjacent 
	to $\image(\sigma_{m_i(a)})$ in the fiber of $\varphi(a)$.
	That is, there exists $b\in\image(\sigma_{m_i(a)})\cap \varphi_0^{-1}(y)$, where $y=\varphi(a)$,
	such that $a$ and $b$ are adjacent in $F_y^0$.
\end{rem}

\begin{notation}
	For a graph $\Gamma=(\Gamma^0,\Gamma^1)$, $|\Gamma^0|\ge2$, and $v\in\Gamma^0$, 
	we denote by $\Gamma - v$ the subgraph of $\Gamma$ given by
	\begin{equation}
	\begin{aligned}
	(\Gamma - v)^0&=\Gamma^0\setminus\{v\}\\
	(\Gamma - v)^1&=\Gamma^1\setminus\{\,e \in\Gamma^1\mid t(e)=v \mbox{ or } o(e)=v\,\}.
	\end{aligned}
	\end{equation}
\end{notation}

Note that if $T=(T^0, T^1)$ is a connected graph, $|T^0|\ge2$ and $v\in T^0$
is a terminal vertex, then $T$ is a tree if and only if $T- v$ is a tree (see \cite[Proposition 9]{ser:trees}).\\
	

We are now ready to show the existence of a tree of representatives for 
the action of a groupoid on a forest (see Definition \ref{def:treeofrepr}).

\begin{prop}
	\label{prop:arboretumrep}
	Let $\caG$ be a groupoid acting without inversion on a forest $F$.
	Then there exists a tree of representatives $(\caG(T),\chi)$ of the
	action of $\caG$ on $F$
	such that $\varphi\big(\chi(\caG_r^{(0)})\big)=\caG^{(0)}$, where $r\in T^0$ is the root of $T$.
\end{prop}
\begin{proof}
	By Theorem \ref{cl:parsec}, there exists a family $\Sigma=\{\sigma_i\}_{i\in I}$, $I\subseteq\N$,
	of partial sections of $F^0$ satisfying (i), (ii) and (iii) of the theorem.
	If $I=\{1\}$, the thesis follows by choosing $T$ to be the tree with one vertex $v$ and no edges. 
	Then $\caG_v=\Stab_{\caG}(\sigma_1)$ (cf. Definition \ref{def:stab})
	is the groupoid associated to the vertex $v$.
	Moreover, the graph $F'=(F'^0,F'^1)$ given by 
	$
	F'^0=U_1,\quad F'^1=\emptyset
	$
	is the representation of $\caG(T)$ on $F$. 
	If $I=\{1,2\}$, the thesis follows by choosing $T=(T^0,T^1)$ to be the segment tree,
	i.e., $T^0=\{v_1,v_2\}$ and  $T^1=\{e,\bar{e}\}$ as follows
	
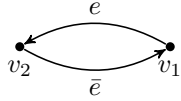
\begin{center}
		\begin{tikzpicture}[->,>=stealth']
		\node[style=circle,fill,inner sep=0pt, minimum size=1.2mm,label=below:{$v_1$}] (v) at (1,0) {};
		\node[style=circle,fill,inner sep=0pt, minimum size=1.2mm,label=below:{$v_2$}] (w) at (-1,0) {};
		
		\path[semithick]
		(v) edge [bend right] node [above] {$e$} (w)
		(w) edge [bend right] node [below] {${\bar{e}}$} (v);
		\end{tikzpicture}
	\end{center}
	with vertex groupoids $\caG_{v_1}=\Stab_{\caG}(\sigma_1)$ and $\caG_{v_2}=\Stab_{\caG}(\sigma_2)$.
	Moreover, the graph $F'=(F'^0,F'^1)$ given by 
	\begin{equation}
		F'^0=U_1\sqcup U_2,\quad F'^1=\{\, e\in F^1 \mid o(e), t(e)\in {F'}^0\,\}
	\end{equation}
	is the representation of $\caG(T)$ on $F$. 
	For $k\ge 2$, $I=\{1,\dots,k\}$, we define the graph with $k$ vertices $T_k=(T_k^0, T_k^1)$ by 
	\begin{equation}
	\begin{aligned}
	T_k^0&=\{\,v_1, \dots, v_k\,\}\\
	T_k^1&=\big\{\, \{v_i,v_j\}\in T^0\times T^0 \mid  v_i\in\image(\sigma_{m_j(v_j)}),\,
	i\ne j,  i,j=1,\dots, k\,\big\},
	\end{aligned}
	\end{equation}
	where $v_{i}\in\image(\sigma_{i})$ is adjacent to $v_{j}\in\image( \sigma_{m_{i}(v_i)})$
	and $m_i$ is the function defined in Remark \ref{rem:mn}.
	Then for each $k\ge 2$ one has that 
	\begin{equation}
	\begin{aligned}
	T_k^0&= T_{k-1}^0\sqcup\{v_k\}\\
	T_k^1&= T_{k-1}^1\sqcup\{\,e\in T^1\mid o(e)=v_k \mbox{ or } t(e)=v_k\,\}.
	\end{aligned}
	\end{equation}
	Thus, for $k\ge 1$, one has that $T_k$ is a tree.
	Finally, we put 
	\begin{equation}
	\label{eq:T}
	T=\bigcup_{i\in I} T_i.
	\end{equation}
	We associate to each $v_i\in T^0$ the groupoid 
	\begin{equation}
	\label{eq:vertexgpd}
	G_{v_i}=\Stab_{\caG}(\sigma_i)
	\end{equation}
	Moreover, for $e\in T^1$ we define the partial section 
	$\sigma_e\colon A\subseteq X \to F^1$ by putting $A=\dom(\sigma_{o(e)})\cap\dom(\sigma_{t(e)})$
	and
	\begin{equation}
	\sigma_e(x)=f, 
	\end{equation}
	where $f$ is the unique $f\in F^1$ such that $o(f)\in\sigma_{o(e)}(x), \,t(f)\in\sigma_{t(e)}(x)$, 
	$x\in A$.
	Thus, we associate to each $e\in T^1$ the groupoid 
    \begin{equation}
	\label{eq:edgegpd}
	G_e=\Stab_{\caG}(\sigma_e).
	\end{equation}
	Finally, we put
	\begin{equation}
		F'^0=V ,\quad F'^1=\{\,e\in F^1\mid o(e),t(e)\in F'^{0}\, \}.
    \end{equation} 
    Since $\tilde{V}=F^0$, one has that the graph $F'=(F'^0,F'^1)$ is the representation of a tree of representatives of the action of $\caG$ on $F$.
    By construction, if $v,w\in F^0$ are in the same orbit, then either $v\in F'^0$ or $w\in F'^0$.
\end{proof}

Note that by definition of $T$, one has an orientation $T_+^1\subseteq T^1$ of $T$.

\begin{rem}
	Let $\caG$ be a groupoid acting on a forest $F$. Then there exists a subforest $F'$ of $F$ 
	such that the vertices of $F'$ are a fundamental domain for the action of $\caG$ on $F^0$.
	In particular, the vertices of $F'$ are a fundamental domain for the equivalence relation $\caR_{\caG}$ generated by the action of $\caG$ on $F^0$. 
	Then the restriction $\pi|_{F'}\colon F'\to Q(\caG,F)$
	is injective on the vertex set
	and its image $\pi|_{F'}(F')$ is a maximal subforest of $Q(\caG,F)$, i.e.,
	$\pi|_{F'^0}(F'^0)=Q(\caG,F)^0$.
\end{rem}

In classical Bass-Serre theory, given a group acting without inversion on a tree, one defines a tree of groups by considering 
a maximal subtree of the quotient graph and the stabilizers of its vertices and edges.
A tree of representatives has the same role in this context.
However, the tree of representatives $(\caG(T),\chi)$ associated to a groupoid action on a forest $F$
does not contain all the information we need: 
in general, 
the $\caG$-orbits of the $\chi(G_e)$'s, $e\in T^1$, do not cover $F^1$.
Thus, we need the following definition.

\begin{defi}
	\label{def:desing}
	Let $\caG$ be a groupoid acting on a forest $F$. 
	A {\it desingularization} $\mathfrak{D}(\caG,F)$ of the action of $\caG$ on $F$ 
	(or a desingularization of $\caG\qgraph F$)
	consists of 
	\begin{itemize}
		\itemsep0em
		\item[(D1)] a connected graph $\Gamma$ with orientation $\Gamma_+^1$;
		\item[(D2)] a graph of groupoids $\caG(\Gamma)$ based on $\Gamma$;
		\item[(D3)] a maximal rooted subtree $T\subseteq\Gamma$;
		\item[(D4)]	a tree of representatives $(\caG(T),\chi)$;
		\item[(D5)] a partial section $\sigma_e\colon X_e\subseteq\caG^{(0)}\to F^1$  
		together with a family of elements $g_e=\big\{\,g_{e,x}\,\big\}_{x\in\varphi(\image(\sigma_e))}\subseteq\caG$ 
		for any $e\in\Gamma_+^1\setminus T^1$;
	\end{itemize}
	satisfying the following properties:
	\begin{itemize}
		\item[(i)] the tree of groupoids $(\caG(T),\chi)$ induced on $T$ is a tree of representatives for
		the action of $\caG$ on $F$ where $\varphi\big(\,\chi\big(T_r^{(0)}\big)\,\big)=\caG^{(0)}$ and $r\in T^0$ is the root of $T$;
		\item[(ii)] for each $e\in\Gamma_+^1\setminus T^1$ and 
		for $\eps\in\image(\sigma_e)$ one has that $o(\eps)\in\chi\big(\caG_{o(e)}^{(0)}\big)$
		and $g_{e,\varphi(t(\eps))}\cdot t(\eps)\in\chi\big(\,\caG_{t(e)}^{(0)}\,\big)$;
		\item[(iii)] for each $e\in\Gamma_+^1\setminus T^1$, 
		the maps
		\begin{align}
		\alpha_{\bar{e}}&\colon\caG_e\to \caG_{o(e)},\notag\\
		\alpha_e&\colon\caG_e\to\caG_{t(e)}\notag
		\end{align}	
		are given by inclusion and conjugation by $g_e$ (cf. Definition \ref{def:conj}) composed with inclusion, respectively;
		\item[(iv)] the saturations with respect to $\caR_{\caG}$
		of the $\chi(\caG_e)$'s, $e\in\Gamma^1$, 
		form a partition of $F^1$.
	\end{itemize}
\end{defi}

We prove that there exists a desingularization for any groupoid action without inversion on a forest.
This result is analogous to the construction of a graph of groups 
associated to a group action without inversion on a tree.
The main difference  is that in the classical Bass-Serre theory the underlying graph $\Lambda$ of  
the graph of groups $G(\Lambda)$
associated to a group action without inversion on a tree $T$ is the quotient graph $\Lambda=G\qgraph T$,
while in this context there
is no canonical graph underlying the desingularization.

\begin{thm}
	\label{thm:desingular}
	Let $\caG$ be a groupoid acting on a forest $F$. 
	Then there exists a desingularization of the action of $\caG$ on $F$.
\end{thm}
\begin{proof}
	By Proposition \ref{prop:arboretumrep}, there exists a tree $T$ (see \eqref{eq:T}) and a tree of representatives
	$(\caG(T),\chi)$ such that $\varphi\big(\chi(\caG_r^{(0)})\big)=\caG^{(0)}$, 
	where $r\in T^0$ is the root of $T$.
	Let $F'=({F'}^0, {F'}^1)\subseteq F$ be the representation of $T^\mathrm{pi}$ on $F$.
	Then one has
	\begin{equation}
	F^0=\bigsqcup_{v\in T^0} \mu^0\big(\,\caG,\,\chi\big(\caG_v^{(0)}\big)\,\big),
	\end{equation}
	i.e., the $\caG$-orbits of the $\chi(\caG_v^{(0)})$'s, $v\in T^0$, form a partition of $F^0$.
	Let  
	\begin{equation}
	S=\bigsqcup_{e\in T^1} \mu^1\big(\,\caG, \,\chi\big(\caG_e^{(0)}\big)\,\big)
	\end{equation}
	i.e., $S$ is the union of the $\caG$-orbits of the $\chi(\caG_e)$'s, $e\in T^1$.
	Let $C=F^1\setminus S$.
	If $C=\emptyset$, then we put $\Gamma=T$ and $\caG(T)$ is the graph of groupoids required
	(cf. Proposition \ref{prop:arboretumrep}).
	Suppose that $C\ne\emptyset$ and let 
	\begin{equation}
	C'=\big\{\,e\in C\mid o(e)\in {F'}^0\,\big\}.
	\end{equation}
	Then $C'\ne\emptyset$ since the fibers of $F$ are connected graphs 
	and since the action of $\caG$ preserves the distances on $F^0$.
	Thus, we construct a family of partial sections of $C$ such that
	the saturations of their images
	form a partition of $C$.
	Let $v\in T^0$ and let $\sigma_v$ be the partial section associated to $v$.
	For all $w\in T^0$, let 
	\begin{equation}
	C_{v,w}=\big\{\,e\in C\mid o(e)\in\image(\sigma_v),\, t(e)\in S\big(\chi(\caG_w)\big) \,\big\}.
	\end{equation}
	where $S(\chi(\caG_w))$ denotes the saturation of $\chi(\caG_w)$ with respect to $\caR_\caG$.
	If $C_{v,w}\ne\emptyset$, then it is fibered on $\varphi\circ o\, (C_{v,w})\subseteq\caG^{(0)}$,
	where $\varphi$ is the momentum map.
	Then by Theorem \ref{cl:parsec} there exists a family $\{\sigma_i\}_{i\in I}$, $I\subseteq\N$, of partial sections such that
	\begin{equation}
	C_{v,w}=\bigsqcup_{i\in I} \,S\big(\image(\sigma_i)\big)
	\end{equation}
	gives a partition of $C_{v,w}$.
	For each $i$ in the decomposition of $C_{v,w}$ we define a new edge $f_i(v,w)\in\Gamma^1$
	such that $o(f_i(v,w))=v$, $t(f_i(v,w))=w$.
	Thus, $\Gamma=(\Gamma^0,\Gamma^1)$
	is the graph given by 
	\begin{equation}
	\begin{aligned}
	\Gamma^0&=T^0,\\ 
	\Gamma^1&=T^1\sqcup\{\,f_i(v,w),\, \bar{f}_i(v,w) \mid i\in I, \,v,w\in T^0 \,\}.
	\end{aligned}
	\end{equation}
	We put
	\begin{equation}
	\label{eq:ups+-}
	\Upsilon_+=\{\,f_i(v,w) \mid i\in I,\, v,w\in T^0\,\},\quad 
	\Upsilon_-=\{\,\bar{f}\mid f\in\Upsilon_+\,\},
	\end{equation}
	and $\Gamma_+^1= T_+^1\sqcup \Upsilon_+$.
	Thus, we have (D1) and (D3). 
	It remains to construct (D2), (D4), (D5).
	For $v\in T^0$ and $e\in T^1$, we put $\caG_v=\Stab_\caG(\sigma_v)$
	and $\caG_e=\Stab_{\caG}(\sigma_e)$ as in \eqref{eq:vertexgpd}
	and \eqref{eq:edgegpd}, respectively.
	For $f_i\in \Upsilon_+$, 
	we associate to $f_i$ the groupoid 
	\begin{equation}
	\caG_{f_i}=\caG_{\bar{f}_i}= \Stab_{\caG}(\sigma_i).
	\end{equation}
	Then $\caG_{f_i}$ injects naturally into $\caG_v$, i.e., 
	\begin{equation}
	\alpha_{\bar{f_i}}\colon
	\caG_{f_i}=\Stab_{\caG}(\sigma_{i})\hookrightarrow \caG_v=\Stab_{\caG}(\sigma_v)
	\end{equation}
	is given by inclusion.
	Moreover, by construction one has that for $\eps_i\in\image(\sigma_i)$,
	there exists $g_{f_i,\varphi(t(\eps_i))}\in \caG$ such that 
	\begin{equation}
	\mu^0\big(\,g_{f_i,\varphi(t(\eps_i))}, t(\eps_i)\,\big)\in {F'}^0.
	\end{equation}
	Let $\caG_w=\Stab_{\caG}(\sigma_w)$ as in \eqref{eq:vertexgpd}.
	Then the groupoid $\caG_{f_i}$ injects into $\caG_w$ via the conjugation by 
	$g_{f_i}=\{\,g_{f_i,\varphi(t(\eps_i))}\,\}_{\eps_i\in\image(\sigma_i)}$ as follows.
	Put $\sigma_\tau:=t\circ\sigma_{i}\colon\dom(\sigma_i)\to F^0$. 
	Then $\caG_{f_i}=\Stab_{\caG}(\sigma_{i})$ injects naturally into 
	$\Stab_{\caG}(\sigma_\tau)$,
	i.e., the map 
	$\iota\colon\caG_{f_i} \hookrightarrow \Stab_{\caG}(\sigma_\tau)$ is given by inclusion.
	Since the subgroupoid $\Stab_\caG(\sigma_\tau)\subseteq\caG$ and $\caG_w\subseteq\caG$
	have the same unit space,
	the equivalence of categories $\Stab_\caG(\sigma_\tau)\to\caG_w$ in Definition \ref{def:conj}
	is an identity on the unit space $\Stab_\caG(\sigma_\tau)^{(0)}$, 
	and the conjugation map is given by
	\begin{equation}
	\begin{aligned}
	i_{g_{f_i}}\colon\Stab(\sigma_\tau)&\to\caG_w\\
	\gamma&\mapsto g_{f_i,r(\gamma)}\,\gamma\,g_{f_i,s(\gamma)}^{-1}.
	\end{aligned}
	\end{equation}
	Since the identity functor is an injective equivalence of categories, by Lemma \ref{lem:injconj}
	one has that $i_{g_{f_i}}$ is injective. Hence, the composition
	\begin{equation}
	\label{eq:alphaconj}
	\alpha_{f_i}\colon \caG_{f_i}\xrightarrow[]{\iota}\Stab_{\caG}(\sigma_\tau)
	\xrightarrow[]{i_{g_{f_i}}} \caG_w,
	\end{equation}
	where the first map is given by inclusion and the second map by conjugation by $g_{f_i}$,
	is an injective homomorphism of groupoids.
	Thus, we have (D2), (D4) and (D5).
\end{proof}


\section{The fundamental groupoid of a graph of groupoids}
\label{s:fundgpd1}

In the previous section we have proved that given a groupoid $\caG$ acting without inversion of edges
on a forest $F$, there exists a desingularization 
$\mathfrak{D} (\caG,F)$ 
associated to the action of $\caG$ on $F$ (cf. Theorem \ref{thm:desingular}),
and hence a graph of groupoids associated to the action of $\caG$ on $F$.
We now want to prove that it is possible to reconstruct $\caG$ in terms 
of the vertex end edge groupoids of the graph of groupoids  $\caG(\Gamma)$.
In particular, we will prove that $\caG$ is isomorphic to the 
fundamental groupoid of the graph of groupoids $\caG(\Gamma)$.

\subsection{The fundamental groupoid of a graph of groupoids}
\label{ss:fundgpd}
Let $\caG(\Gamma)$ be a graph of groupoids and let $T\subseteq\Gamma$
be a maximal rooted subtree.
Let $\caH_e=\alpha_e(\caG_e)$ and let $\phi_e\colon\caH_{\bar{e}}\to\caH_e$ be the map
given by $\phi_e(g)=\alpha_e(\alpha_{\bar{e}}^{-1}(g))$.
Since the groupoids 
$\caH_e$ are wide subgroupoids of $\caG_{t(e)}$ for all $e\in\Gamma^1$,
the map 
\begin{equation}
\phi_e^0=\phi_e|_{\caG_{o(e)}^{(0)}}\colon \caG_{o(e)}^{(0)}\to\caG_{t(e)}^{(0)}
\end{equation}
is a bijection for all $e\in\Gamma^1$.
Then the maps $\phi_e^0$'s induce a rooted tree of partial isomorphisms $T^{\mathrm{pi}}$
which defines an equivalence relation $\sim$ (see \eqref{eq:relpi}) on 
$Y=\sqcup_{v\in\Gamma^0} \,\caG_v^{(0)}$.
Let $\tilde{Y}=Y/\sim$ and $\pi\colon Y\to \tilde{Y}$ be the canonical projection.
Let $r\in T^0$ be the root of the tree $T$. 
Then 
\begin{equation}
\label{eq:X}
X=\caG_r^{(0)}
\end{equation}
is a fundamental domain for the equivalence relation $\sim$
and the map
\begin{equation}
\tau_0:=\pi|_X^{-1}\colon\tilde{Y}\to X
\end{equation}
is bijective. Let 
\begin{equation}
\tau=\tau_0\circ\pi\colon Y\to X 
\end{equation}
the map assigning to each $y\in Y$ the representative
$x\in X$ such that $y\sim x$.

\begin{defi}
	The {\it fundamental groupoid} $\pi_1(\caG(\Gamma))$ of the
	graph of groupoids $\caG(\Gamma)$
	is the groupoid defined as follows (see \cite[\S 3]{hig64} for presentation of groupoids).
	It has unit space $X$ defined in \eqref{eq:X} and generators
	\begin{equation}
	\label{eq:genpi}
	\Bigg(\bigsqcup_{v\in\Gamma^0} \,\caG_v\Bigg)\,\sqcup\,
	\Bigg(\,\bigsqcup_{\substack{e\in\Gamma^1,\\
			x\in\caG_{o(e)}^{(0)}}} (\phi_e(x),x)\,\Bigg).
	\end{equation}
	Note that for $z\in \caG_e^{(0)}$ one has that 
	$(\alpha_e (z), \alpha_{\bar{e}}(z))$ is a generator, since
	$\phi_e(\alpha_{\bar{e}}(z))=\alpha_e\circ\alpha_{\bar{e}}^{-1} (\alpha_{\bar{e}}(z))=\alpha_e(z)$.
	The range and source maps $r,s\colon\pi_1(\caG(\Gamma))\to X$ are defined by 
	\begin{align}
	r(g)&=\tau(r'(g)), \quad r((\phi_e(x),x))=\tau(\phi_e(x)),\\
	s(g)&=\tau(s'(g)), \quad s((\phi_e(x),x))=\tau(x),
	\end{align}
	where $r'$ and $s'$ denotes the range and source maps in $\caG_v$, $v\in\Gamma^0$,
	and multiplication is given by concatenation.
	The inverse map is given by the inverse map in $\caG_v$ for elements $g\in\caG_v$,
	$v\in\Gamma^0$, and by $(\phi_e(x),x)^{-1}=(x,\phi_e(x))$ for $e\in\Gamma^1$, $x\in\caG_{o(e)}^{(0)}$.
	Moreover, the generators are subject to the relations in the vertex groupoids together with relations
	\begin{align}
	\label{eq:R1}
	&(\phi_e(x),x)=\iid_{\tau(x)} \mbox{ for } e\in T^1, \,x\in\caG_{o(e)}^{(0)};\tag{R1}\\
	&\big(\alpha_e(r(g)), \alpha_{\bar{e}}(r(g))\big)\cdot \alpha_{\bar{e}}(g)\cdot
	\big(\alpha_{\bar{e}}(s(g)),\alpha_e(s(g))\big)
	=\alpha_e(g) \mbox{ for }g\in\caG_e, e\in\Gamma^1.\tag{R2}
	\end{align}
\end{defi}

\begin{defi}
	We call the {\it free groupoid} $\caF(\caG(\Gamma))$ the groupoid with the same 
	generators
	as $\pi_1(\caG(\Gamma))$, i.e., the generators in \eqref{eq:genpi}, subject to the relations in the vertex groupoids together with the relation (R2).
\end{defi}

\begin{notation}
	In order to simplify calculations with fundamental groupoids elements, 
	we will use the following notation:
	\begin{equation*}
	\phi_e(x) e=ex:=(\phi_e(x),x),
	\end{equation*}
	for $x\in\caG_{o(e)}$, $e\in\Gamma^1$.
	Thus, we may rewrite relations (R1) and (R2) as follows:
	\begin{align}
	&	\phi_e(x) e=ex=\iid_{\tau(x)},\quad \mbox{ for }x\in\caG_{o(e)}, \,e\in T^1\tag{R1}.\\
	&e\,\alpha_{\bar{e}}(r(g)) \cdot \alpha_{\bar{e}}(g)\cdot \alpha_{\bar{e}}(s(g))\,\bar{e}=\alpha_e(g)
	\quad \mbox{ for } g\in\caG_e,\, e\in\Gamma^1.\tag{R2}
	\end{align}
	Since for any groupoid $G$ and any morphism $\gamma\in G$ one has  $r(\gamma)\gamma=\gamma=\gamma s(\gamma)$, we put
	\begin{equation}
	g_1 e g_2= g_1\cdot s(g_1) \,e\, r(g_2)\cdot g_2
	\end{equation}
	for $g_1\in\caG_{t(e)}$, $g_2\in\caG_{o(e)}$, $s(g_1)=\phi_e(r(g_2))$.
	With this notation, (R2) becomes
	\begin{equation}
	\label{eq:reledges}
	\tag{R2}
	e \, \alpha_{\bar{e}}(g)\,\bar{e}=\alpha_e(g),
	\end{equation}
	for $g\in\caG_e, e\in\Gamma^1$.
\end{notation}

We can extend our definition of a path in $\Gamma$ as follows.

\begin{defi}
	A {\it graph of groupoids word } $w$ in $\caG(\Gamma)$ is a sequence of elements 
	\begin{equation}
	w=g_1 e_1g_2e_2\cdots g_ne_n g_{n+1},
	\end{equation}
	where 
	\begin{itemize}
		\itemsep0em
		\item $e_i\in\Gamma^1$ for all $i=1,\dots,n$;
		\item $o(e_i)=t(e_{i+1})$ for all $i=1,\dots,n-1$;
		\item $g_i\in\caG_{t(e_i)}$ for all $i=1,\dots,n$ and $g_{n+1}\in\caG_{o(e_n)}$;
		\item for all $i=1,\dots,n$ one has $\,\phi_{\bar{e}_i}\big(s(g_i)\big)=r(g_{i+1})$.
	\end{itemize}	
	We say that $w$ has {\it length} $n$. A {\it subword} of $w$ is a sequence
	$g_i e_i\cdots g_{i+k}e_{i+k}$ such that $i\ge 1$, $k\ge 0$ and $i+k\le n$.
\end{defi}

\begin{rem}
	\label{rem:leftransv}
	For all $e\in\Gamma^1$ we choose a {\it left transversal} $\caT_e$ of $\alpha_e(\caG_e)$
	in $\caG_{t(e)}$ containing the identity elements of $\caG_{t(e)}$, i.e., 
	for all $g\in\caG_{t(e)}$ there exist $\tau=\tau(g,e)\in\caT_e$ and 
	$h=h(g,e)\in\caG_{e}$ such that
	\begin{equation}
	\label{eq:transg}
	g=\tau\alpha_e(h).
	\end{equation}
	Then we may represent a graph of groupoids word $w$ by 
	\begin{equation}
	w= g_1 e_1g_2e_2\cdots g_ne_n g_{n+1},
	\end{equation}
	where 
	\begin{itemize}
		\itemsep0em
		\item $e_i\in\Gamma^1$ for all $i=1,\dots,n$;
		\item $t(e_i)=o(e_{i-1})$ for all $i=2,\dots,n$;
		\item $g_i\in\caG_{t(e_i)}$ for all $i=1,\dots,n$ and $g_{n+1}\in\caG_{o(e_n)}$;
		\item for all $i=1,\dots,n$, subwords $g_i e_i g_{i+1}= \tau_i \alpha_{e_i}(h_i) \,e_i \,
		\tau_{i+1}\alpha_{e_{i+1}}(h_{i+1})$
		must satisfy
		\begin{equation*}
		s(\alpha_{\bar{e}_i}(h_i))=r(\tau_{i+1}).
		\end{equation*}
	\end{itemize}
\end{rem}

\begin{defi}
	A graph of groupoids word 
	$w=g_1 e_1\cdots g_n e_n g_{n+1}$ is said to be 
	{\it reduced} if
	\begin{itemize}
		\itemsep0em
		\item $g_i\in\caT_{e_i}$ for all $i=1,\dots,n$ and $g_{n+1}\in\caG_{o(e_n)}$;
		\item if $\bar{e}_i=e_{i-1}$ for some $i=2,\dots,n$, then
		$g_i\notin\image(\alpha_{e_i})$.
	\end{itemize}
	In particular, if $e_1\cdots e_n$ is a path in $\Gamma$ without backtracking, 
	then $w$ is reduced.
\end{defi}

We define the {\it reduction} of words by using the relations (R2).
There are two types of reduction: coset and length.

\begin{defi}
	\label{def:cosred}
	Let $u=geg'$ be a subword of a graph of groupoids word $w$,
	where $g=\tau\alpha_e(h)$ as in \eqref{eq:transg}.
	The {\it coset reduction} $w'$ is the graph of groupoids word obtained from $w$ by replacing 
	$u$ by $v=\tau e g''$, where $g''=\alpha_{\bar{e}}(h)g'$.
\end{defi}
\begin{defi}
	\label{def:lenred}
	Let $u=eg\bar{e}$, $g\in\caH_{\bar{e}}$, be a subword of a graph of groupoids word $w$.
	The {\it length reduction} $w'$ is the graph of groupoids word obtained from $w$ by replacing 
	$u$ by $\phi_e(g)$. 
	Note that when length reduction is applied to a graph of groupoids word, 
	its length is decreased by two.
	Two graph of groupoids words $w$ and $w'$ are said to be {\it equivalent}
	if one can get one from the other by using coset and length reductions.
\end{defi}

Hence a reduced graph of groupoids word is a word in which no coset and no length reduction 
can be applied. 
The reduction process removes subwords $eg\bar{e}$, where $g\in\caH_{\bar{e}}$, 
and moves $\alpha_{e}(h)$ in the decomposition $\tau\alpha_{e}(h)$ to the right.
We prove that $p\in\pi_1(\caG(\Gamma))$ is represented uniquely 
by a reduced graph of groupoids word 
$p=g_1 e_1 g_2 e_2\cdots g_ne_n g_{n+1}$.

By adapting the proofs of \cite[\S 3]{hig:fundgpd} and \cite[Theorem 2.1.7]{moore}, 
one shows the following property.

\begin{thm}
	\label{thm:uniqreducedword}
	Every morphism of $\caF(\caG(\Gamma))$ is represented by a unique reduced
	graph of groupoids word.
\end{thm}

\begin{defi}
	A graph of groupoids word $g=g_1 e_1 g_2 e_2\cdots g_n e_n g_{n+1}$
	is said to be reduced {\it in the sense of Serre} if it satisfies the following:
	if $n=0$, then $g_1$ is not a unit;
	if $n>0$ and $e_{i+1}=\bar{e}_i$, then $g_{i+1}\notin\caH_{\bar{e}_i}$.
\end{defi}
One has the following analogue of Britton's Lemma.

\begin{lem}[Britton's Lemma]
	\label{lem:brit}
	Let $\caG(\Gamma)$ be a graph of groupoids based on $\Gamma$ and let 
	$p=g_1e_1g_2\cdots g_n e_n g_{n+1}$ be a graph of groupoids word
	which is reduced in the sense of Serre.
	Then $p$ is not a unit.
\end{lem}
\begin{proof}
	The case $n=0$ is trivial. Thus, let $n>0$.
	For $i=1,\dots,n$ let $h_i\in\caH_{e_i}$ and $\tau_i\in\caT_{e_i}$ be such that
	\begin{align}
	g_i&=\tau_i h_i,\notag\\
	\phi_{\bar{e}_i}(h_i) \,g_{i+1}&=\tau_{i+1} h_{i+1}.\notag
	\end{align}
	Then by pulling the $h_i$'s to the right
	one has that $p=\tau_1 e_1 \tau_2\cdots \tau_n e_n a$, where
	$\tau_i\in\caT_{e_i}$ and $a\in\caG_{o(e_n)}$.
	If there exists $i$ such that $e_{i+1}=\bar{e}_i$, 
	then one has that $\phi_{\bar{e}_i}(h_i)$ and $h_{i+1}$ are both in $\caH_{\bar{e}_i}$
	and hence it must be $\tau_{i+1}\ne x$, with $x=r(h_{i+1})$.
	Thus, $\tau_1 e_1 \tau_2\cdots \tau_n e_n a$ is not an unit.
\end{proof}

\subsection{The universal cover of a graph of groupoids}

Let $\caG(\Gamma)$ be a graph of groupoids and let $T\subseteq\Gamma$ be a
maximal subtree of $\Gamma$ which induces an equivalence relation as in \eqref{eq:relpi}.
For $v\in\Gamma^0$ and $e\in\Gamma^1$, the groupoids $\caG_v$ and $\caG_e$
induce subgroupoids of $\pi_1(\caG(\Gamma))$.
This allows us to define a {\it universal cover} for $\caG(\Gamma)$.

\begin{defi}
	Let $\caG(\Gamma)$ be a graph of groupoids and 
	let $\caH_e=\alpha_e(\caG_e)\subseteq\caG_{t(e)}$, $e\in\Gamma^1$.
	We fix an orientation $\Gamma_+^1$ of $\Gamma$ and put
	\begin{equation}
	\label{eq:eps}
	\eps(e)=\begin{cases}
	0&\mbox{ if } e\in\Gamma_+^1\\
	1&\mbox{ if } e\in\Gamma^1\setminus\Gamma_+^1.
	\end{cases}
	\end{equation}
	In particular, one has that
	\begin{equation*}
	\eps(\bar{e})=1-\eps(e)
	\end{equation*}
	for all $e\in\Gamma^1$.
	We denote by $|e|$ the edge satisfying
	$\{|e|\}=\{e,\bar{e}\}\cap \Gamma_+^1$.
\end{defi}

Note that by definition, one has that $\caH_{|e|}=\caH_{|\bar{e}|}$ for all $e\in\Gamma^1$.

\begin{notation}
	For $e\in\Gamma^1$, we put 
	\begin{itemize}
		\item[(i)] $e\,\caG_{o(e)}=\{\, \phi_e(r(g)) e g \mid g\in\caG_{o(e)}\,\}$,
		\item[(ii)] $e\,\caG_{o(e)}\,e^{-1}=\{\, \phi_e(r(g)) e g e^{-1}\phi_e(s(g)) \mid g\in\caG_{o(e)}\,\}$.
	\end{itemize}
\end{notation}

\begin{rem}
	\label{rem:welldef}
	Let $e\in\Gamma^1$.
	With the notations above, one has the following
	\begin{align}
	\caH_e &=\caG_{t(e)}\cap\, e\,\caG_{o(e)}\,e^{-1};\tag{a}\\
	\caH_{|e|}&=e^{1-\eps(e)} \,\caH_{\bar{e}} \,e^{\eps(e)-1};\tag{b}\\
	\caH_{|e|}
	&=e^{1-\eps(e)} \,\caH_{\bar{e}} \,e^{\eps(e)-1}\tag{c}\\
	&=e^{1-\eps(e)} \,(\caG_{o(e)}\cap\,e^{-1}\,\caG_{t(e)}\,e)\,e^{\eps(e)-1}\notag\\
	&= e^{1-\eps(e)}  \,\caG_{o(e)} \,e^{\eps(e)-1}\,\cap\, e^{-\eps(e)}\,\caG_{t(e)}\,e^{\eps(e)}.\notag
	\end{align}
\end{rem}

We remind that for $v\in\Gamma^0$ and $p\in\pi_1(\caG(\Gamma))$, the set $p\caG_v$ 
is the set $p\caG_v=\{\,pg\mid g\in\caG_v, s(p)=r(g)\,\}$.
In what follows we simplify the groupoid notation and write 
$\pi_1(\caG(\Gamma))/\caG_v$ to indicate
the set $\big\{\,p\caG_v\mid  p\in\pi_1(\caG(\Gamma)),\, s(p)\in\tau(\caG_v^{(0)})\,\big\}$, 
for $v\in\Gamma^0$.

\begin{defi}
	\label{def:BSforest}
	The {\it Bass-Serre forest $X_{\caG(\Gamma)}$} of the graph of groupoids $\caG(\Gamma)$, also known as its 
	{\it universal cover}, is the graph defined by 
	\begin{align}
	X_{\caG(\Gamma)}^{0}&=\bigsqcup_{v\in\Gamma^0} \pi_1(\caG(\Gamma))/\caG_v [v],\\
	X_{\caG(\Gamma)}^{1}&=\bigsqcup_{e\in\Gamma^1} \pi_1(\caG(\Gamma))/\caH_{|e|} [e].
	\end{align}
	The maps $o,t\colon X_{\caG(\Gamma)}^1\to\Lambda^0$ and 
	$\,\bar{\,}\,\colon X_{\caG(\Gamma)}^1\to X_{\caG(\Gamma)}^1$ 
	are given by
	\begin{align}
	o(p\caH_{|e|})&=pe^{1-\eps(e)}\caG_{o(e)} [o(e)]\,\in\pi_1(\caG(\Gamma))/\caG_{o(e)},\\
	t(p\caH_{|e|})&=pe^{-\eps(e)}\caG_{t(e)}[t(e)]\,\in\pi_1(\caG(\Gamma))/\caG_{t(e)},
	\intertext{and}
	\overline{p\caH_{|e|}[e]}&=p\caH_{|\bar{e}|}[\bar{e}].
	\end{align}
\end{defi}

	Note that  $t$ and $o$ are well defined by Remark \ref{rem:welldef}.
	In fact, suppose that for $p,q\in\pi_1(\caG(\Gamma))$ and $e\in\Gamma^1$ one has 
	that $p\in q\caH_{|e|}$. 
	Then there exists $h\in\caH_{|e|}$ such that $p=qh$.
	By Remark \ref{rem:welldef} (c), there exists $g\in\caG_{o(e)}$
	such that $h=e^{1-\eps(e)} \,g \,e^{\eps(e)-1}$.
	Then one has
	\begin{equation}
	\begin{aligned}
	o(p\caH_{|e|})&= p\,e^{1-\eps(e)}\,\caG_{o(e)}[o(e)]\\
	&=q\,h\,e^{1-\eps(e)}\,\caG_{o(e)}[o(e)]\\
	&=q \,e^{1-\eps(e)} \,g \,e^{\eps(e)-1}\,e^{1-\eps(e)}\,\caG_{o(e)}[o(e)]\\
	&=q \,e^{1-\eps(e)} \,g \,\caG_{o(e)}[o(e)]\\
	&=q \,e^{1-\eps(e)} \,r(g)\,\caG_{o(e)}[o(e)]\\
	&=o(q\caH_{|e|}).
	\end{aligned}	
	\end{equation}
	Moreover, it is easy to see that 
	\begin{align}
	\overline{p\caH_{|e|}[e]}&\ne p\caH_{|e|}[e]\\
	\overline{\overline{p\caH_{|e|}[e]}}&=p\caH_{|e|}[e],
	\intertext{and, by \eqref{eq:eps}}
	o(\overline{p\caH_{|e|}[e]})&=o(p\caH_{|\bar{e}|}[\bar{e}])\notag\\
	&= p \bar{e}^{1-\eps(\bar{e})}\caG_{o(\bar{e})}[o(\bar{e})],\\
	&= p e^{-\eps(e)}\caG_{t(e)}[t(e)],\notag\\
	&= t(p\caH_{|e|}[e])\notag.
	\end{align}
Hence, $X_{\caG(\Gamma)}$ is a graph. 

\begin{rem}
	The fundamental groupoid $\pi_1(\caG(\Gamma))$ acts on $X_{\caG(\Gamma)}$ 
	by left multiplication with momentum map 
	\begin{equation}
	\label{eq:varphibs}
	\begin{aligned}
	\tivarphi\colon X_\caG^{0}&\to\pi_1(\caG(\Gamma))^{(0)}\\
	p\caG_v&\mapsto r(p)
	\end{aligned}
	\end{equation}
	(see Definition \ref{def:gpdaction}).
	That is, 
	for $p,p'\in\pi_1(\caG(\Gamma))$, $q\caG_v[v]\in X_{\caG(\Gamma)}^0$, 
	$q'\caH_e[e]\in X_{\caG(\Gamma)}^1$ 
	with $s(p)=\tivarphi(q\caG_v)=r(q)$ and $s(p')=\tivarphi\circ o (q'\caH_e)=r(q')$ 
	one has that 
	\begin{align}
	\mu^0(p,q\caG_v[v])&= pq\caG_v[v],\\
	\mu^1(p',q'\caH_{|e|}[e])&= p'q'\caH_{|e|}[e],
	\end{align}
	where $pq\in\pi_1(\caG(\Gamma))$ is the reduced graph of groupoids word obtained from 
	the concatenation of $p$ and $q$.
\end{rem}

Note further that for $x\in\pi_1(\caG(\Gamma))^{(0)}$ and for all $e\in T^1$ one has that
\begin{equation}
\label{eq:xedge1}
o(x\caH_{|e|}[e])=x\caG_{o(e)}[o(e)],\quad 
t(x\caH_{|e|}[e])=x\caG_{t(e)}[t(e)],
\end{equation}
and for all $e\in \Gamma_+^1\setminus T^1$ one has that
\begin{equation}
\label{eq:xedge2}
o(x\caH_{|e|}[e])= xe\caG_{o(e)}[o(e)] ,\quad 
t(x\caH_{|e|}[e])= x\caG_{t(e)}[t(e)].
\end{equation}

\begin{rem}
	\label{rem:bsforest}
	The Bass-Serre forest $X_{\caG(\Gamma)}$ is fibered on $\pi_1(\caG(\Gamma))^{(0)}$
	via the map $\tivarphi$ defined in \eqref{eq:varphibs}.
	For each unit $x\in\pi_1(\caG(\Gamma))^{(0)}$, 
	we denote by $xX_{\caG(\Gamma)}$ the subgraph of $X_{\caG(\Gamma)}$
	which is fibered on $\{x\}$, i.e., 
	\begin{equation}
	\begin{aligned}
	xX_{\caG(\Gamma)}^0&=\tivarphi^{-1}(x),\\
	xX_{\caG(\Gamma)}^{1}&=o^{-1}(\tivarphi^{-1}(x)).
	\end{aligned} 
	\end{equation}
	
	Clearly $xX_{\caG(\Gamma)}$ and $yX_{\caG(\Gamma)}$ are disconnected for all $x\ne y$, $x,y\in\pi_1(\caG(\Gamma))^{(0)}$.	
	Since each vertex in $xX_{\caG(\Gamma)}$ starts with $x$ and
	each vertex in $yX_{\caG(\Gamma)}$ starts with $y$, 
	i.e., for $p\caG_v\in X_x$ and $q\caG_w\in X_y$ one has $r(p)=x$ and $r(q)=y$,
	there is no edge connecting $xX_{\caG(\Gamma)}$ and $yX_{\caG(\Gamma)}$.
	In fact, the map $\tivarphi$ is such that $\tivarphi\circ o =\tivarphi\circ t$ by definition.
	Moreover, $xX_{\caG(\Gamma)}$ is a tree (cf. Proposition \ref{prop:tree}).
	It follows that $X_{\caG(\Gamma)}$ is a forest, so that the terminology used in Definition \ref{def:BSforest} is justified.
\end{rem}

\begin{prop}
	\label{prop:tree}
	$xX_{\caG(\Gamma)}$ is a tree for all $x\in\pi_1(\caG(\Gamma))^{(0)}$.
\end{prop}
\begin{proof}
	Fix $x\in\pi_1(\caG(\Gamma))^{(0)}$.
	First we need to prove that $xX_{\caG(\Gamma)}$ is a connected graph.
	Let $\Xi$ be the smallest subgraph of $xX_{\caG(\Gamma)}$ containing 
	$\{\,x\caH_{|e|}[e]\mid e \in\Gamma^1\,\}$.
	Then $\Xi$ is connected by \eqref{eq:xedge1}.
	and \eqref{eq:xedge2}.
	Moreover, $xX_{\caG(\Gamma)}= x\pi_1(\caG(\Gamma))x\ast\Xi$, where
	$x\pi_1(\caG(\Gamma))x$ is the isotropy group of $\pi_1(\caG(\Gamma))$ at $x$,
	i.e., 
	\begin{equation*}
	x\pi_1(\caG(\Gamma))x=\{p\in\pi_1(\caG(\Gamma))\mid s(p)=x=r(p)\}.
	\end{equation*}
	By Proposition \ref{prop:cayley},
	it suffices to show that there exists a generating system $\caS\subseteq \pi_1(\caG(\Gamma))$
	such that $\Xi\cup a\ast\Xi$ is connected for all $a\in x\caS x$.
	Then one has that the graph
	\begin{equation}
	\Xi\,\cup\, a_1\ast\Xi\,\cup\, a_1 a_2\Xi\,\cup\,\cdots\,\cup a_1 a_2\cdots a_n\ast\Xi
	\end{equation}
	is connected for all $a_1,\dots,a_n\in x\caS x$ by induction on $n$.
	It suffices to prove the claim for $a\in x\caS x$. Choosing 
	\begin{equation}
	\caS= \Bigg(\bigsqcup_{v\in\Gamma^0} \caG_v \,\Bigg)\,\sqcup \,
	\Bigg(\bigsqcup_{
		\substack{e\in\Gamma_+^1,\, 
			y\in\caG_{o(e)}^{(0)} }}
	\phi_e(y) e y\,\Bigg).
	\end{equation}
	one verifies the claim for $a\in x\caS x$.
	Since $\Xi$ is connected by construction and $a\ast \Xi$ is connected
	because the action of $\pi_1(\caG(\Gamma))$ preserves the distances, it is sufficient to prove that 
	there exists a vertex lying in both $\Xi$ and $a\ast\Xi$, i.e., 
	$\Xi^0\cap a\ast\Xi^0\ne \emptyset$.
	By definition, for any $a\in \caS $ there exists a vertex 
	$p\caG_v$, $v\in\Gamma^0$, such that
	$a\ast p\caG_v$ is defined and $a\ast p\caG_v=p\caG_v$.
	Fix $a\in x\caS x$, $v\in\Gamma^0$ such that $a\ast x\caG_v=x\caG_v$.
	Then one has that  $x\caG_v\in \Xi^0\cap a\ast\Xi^0\ne \emptyset$, which proves that 
	$\Xi\cup a\ast\Xi$ is connected.
	It remains to prove that $xX_{\caG(\Gamma)}$ does not contain non-trivial reduced paths from $v$ to $v$, 
	$v\in xX_{\caG(\Gamma)}^{0}$. 
	Let 
	\begin{equation}
	\mathfrak{p}=(p_1\caH_{|e_1|}[e_1])\,(p_2\caH_{|e_2|}[e_2])\,\cdots\, (p_n\caH_{|e_n|}[e_n]),
	\quad n\ge1
	\end{equation}
	be such a path. Put $v_i=o(e_i)$.
	In particular, one has  $v_n=o(e_n)=t(e_1)$.
	Then one has
	\begin{align}
	o( p_n\caH_{|e_n|}[e_n])&= p_n e_n^{1-\eps(e_n)}\caG_{v_n}[v_n]
	= p_1 e_1^{-\eps(e_1)}\caG_{v_n}[v_1]=t( p_1\caH_{|e_1|}[e_1]) ,\notag\\
	o(p_1\caH_{|e_1|}[e_1])&= p_1 e_1^{1-\eps(e_1)}\caG_{v_1}[v_1]
	= p_2 e_2^{-\eps(e_2)}\caG_{v_1}[v_2]=t( p_2\caH_{|e_2|}[e_2] ) ,\notag\\
	\vdots&\notag\\
	o( p_{n-1}\caH_{|e_{n-1}|}[e_{n-1}])&= p_{n-1} e_{n-1}^{1-\eps(e_{n-1})}\caG_{v_{n-1}}[v_{n-1}]
	= p_n e_n^{-\eps(e_n)}\caG_{v_{n-1}}[v_n]=t(p_n\caH_{|e_n|}[e_n]).\notag
	\end{align}
	In particular, putting $q_i=p_i e_i^{-\eps(e_i)}$, there exist elements $a_i\in\caG_{v_i}$ with 
	$r(a_i)=\phi_{\bar{e}_1}(s(q_i))$ such that
	\begin{align}
	q_n e_n a_n&=q_1\notag\\
	q_1 e_1 a_1&=q_2\notag\\
	\vdots&\notag\\
	q_{n-1} e_{n-1} a_{n-1}&=q_n.\notag
	\end{align}
	In particular, one has 
	\begin{align}
	s(q_n) e_n a_n&=q_n^{-1} q_1\notag\\
	s(q_1) e_1 a_1&=q_1^{-1} q_2\notag\\
	\vdots&\notag\\
	s(q_{n-1}) e_{n-1} a_{n-1}&=q_{n-1}^{-1} q_n.\notag
	\end{align}
	Then one has
	\begin{equation}
	s(q_1) e_1 a_1 e_2 a_2 \cdots e_n a_n= s(q_1).
	\end{equation}
	Thus, by Lemma \ref{lem:brit}, $\mathfrak{q}:=s(q_1) e_1 a_1 e_2 a_2 \cdots e_n a_n$
	is not reduced.
	Hence there exists $i\in\{1,\dots,n-1\}$ such that $e_{i+1}=\bar{e}_i$
	and $a_i\in\image(\alpha_{\bar{e}_i})=\caH_{\bar{e}_i}$, i.e., there exists $b\in\caG_{e_i}$
	such that $a_i=\alpha_{\bar{e}_i}(b)$ and
	\begin{equation}
	e_i a_i e_{i+1}= e_i \alpha_{\bar{e}_i}(b) \bar{e}_i = \alpha_{e}(b)\in\caG_{t(e_i)}=\caG_{v_{i-1}}
	\end{equation}
	Here we put $v_0=v_n$. Then one has
	\begin{equation}
	\begin{aligned}
	p_{i+1}\caH_{|e_{i+1}|}[e_{i+1}]
	&= q_{i+1} e_{i+1}^{\eps(e_{i+1})}\caH_{|e_i|}[\bar{e}_i]\\
	&= q_i e_i a_i \bar{e}_i^{\eps(\bar{e}_i)}\caH_{|e_i|}[\bar{e}_i]\\
	&= p_i e_i^{1-\eps(e_i)} a_i \bar{e}_i^{1-\eps(e_i)}\caH_{|e_i|}[\bar{e}_i]\\
	&= p_i e_i^{1-\eps(e_i)} a_i e_i^{\eps(e_i)-1}\caH_{|e_i|}[\bar{e}_i].
	\end{aligned}
	\end{equation}
	Note that
	\begin{equation}
	\begin{aligned}
	e_i^{1-\eps(e_i)} a_i e_i^{\eps(e_i)-1}&=
	\begin{cases}
	e_i a_i e_i^{-1}&\mbox{ if } |e_i|=e_i,\\
	a_i &\mbox{ if } |e_i|=\bar{e}_i
	\end{cases}
	=\begin{cases}
	\alpha_{e_i}(b) &\mbox{ if } |e_i|=e_i,\\
	a_i &\mbox{ if } |e_i|=\bar{e}_i.
	\end{cases}
	\end{aligned}
	\end{equation}
	Since $\alpha_{e_i}(b)\in\caH_{e_i}$ and $a_i=\alpha_{\bar{e}_i}(b)\in\caH_{\bar{e}_i}$, 
	one has that $e_i^{1-\eps(e_i)} a_i e_i^{\eps(e_i)-1}\in\caH_{|e_{i}|}$
	by Remark \ref{rem:welldef} (b).
	Hence one has 
	\begin{equation}
	p_{i+1}\caH_{|e_{i+1}|}[e_{i+1}]
	=p_{i}\caH_{|e_{i}|}[\bar{e}_{i}]
	=\overline{p_{i}\caH_{|e_{i}|}[e_{i}]}.
	\end{equation}
	Thus, $\mathfrak{p}$ is not reduced, a contradiction, and this yields the claim.
\end{proof}

\section{The structure theorem}

\label{s:structhm}

In this section we collect all the results of the previous sections. 
Let $\caG$ be a groupoid acting on a forest $F$  without inversion of edges. 
Then we have seen in Section \ref{s:gractforest}
that one has a desingularization $\mathfrak{D}(\caG,F)$,
i.e., a graph of groupoids $\caG(\Gamma)$ based on a graph $\Gamma=T\sqcup\Upsilon_+\sqcup \Upsilon_-$,
where $T\subseteq \Gamma$ is a maximal subtree 
such that $\caG(T)$ is a tree of representatives for the action of $\caG$ on $F$
and $\Upsilon_+$ and $\Upsilon_-$ are as in \eqref{eq:ups+-},
and a family of groupoid elements 
$g_e=\{g_{e,x}\}_{x\in\varphi(t(\image(\sigma_e)))}$ for each
$e\in\Gamma^1$,
such that 
\begin{itemize}
	\itemsep0em
	\item[(i)] $\caG_v=\Stab_{\caG}(\sigma_v)$ for $v\in\Gamma^0$;
	\item[(ii)] $\caG_e=\caG_{\bar{e}}=\Stab_{\caG}(\sigma_e)$ for $e\in\Gamma^1$;
	\item[(iii)] $\alpha_e\colon\caG_e\to\caG_{t(e)}$ is given by inclusion 
	for $e\in T^1\sqcup\Upsilon_+$;
	\item[(iv)] $\alpha_{e}\colon\caG_e\to\caG_{t(e)}$
	is given by conjugation by $g_e$ for $e\in \Upsilon_-$ (see \eqref{eq:alphaconj}). 
\end{itemize}

We put $g_e=\{\iid_x\}_{x\in\varphi(\image(\sigma_e))}$ for $e\in T^1$ and 
$g_e=g_{\bar{e}}^{-1}$ for $e\in\Upsilon_+$.

So by definition for all $e\in\Gamma_+^1$ one has a commutative diagram 

\begin{center}
	\begin{tikzcd}[column sep=normal, row sep=large]
		&
		&\caG_e \ar[dl, "\alpha_{\bar{e}}",labels=left] \ar[dr, "\alpha_{e}",labels=right]
		&
		&[1.0em]\\
		&\caG_{o(e)} \ar[d, hook]
		& 
		& \caG_{t(e)}  \ar[d, hook]
		&[1.0em]\\
		& \caG\ar[rr, "i_{g_e}"]
		&
		& \caG
	\end{tikzcd}
\end{center}

\begin{prop}
	For $\caG$, $F$, $\mathfrak{D}(\caG,F)$ and $\{g_e\}_{e\in\Gamma^1}$ as above,
	the assignment 
	\begin{equation}
	\psi_\circ\colon \bigsqcup_{\substack{e\in\Gamma_+^1\\
			x\in\caG_{o(e)}^{(0)}}} (\phi_e(x), x) 
	\,\sqcup\,
	\bigsqcup_{v\in\Gamma^0} \caG_v\,
	\longrightarrow\,\caG
	\end{equation}
	given by 
	\begin{equation}
	\begin{aligned}
	\label{eq:psig}
	\psi_\circ(g)&=g, \quad &&g\in\caG_v,\\
	\psi_\circ((\phi_e(x),x))&= g_{e,x}, \quad &&e\in\Gamma^1, \,x\in\caG_{o(e)}^{(0)},
	\end{aligned}
	\end{equation}
	defines a groupoid homomorphism
	\begin{equation}
	\psi\colon\pi_1(\caG(\Gamma))\longrightarrow \caG.
	\end{equation}
\end{prop}
\begin{proof}
	By definition, $\psi_\circ$ satisfies the relation
	\begin{equation}
	\psi_\circ\big((\phi_e(x),x)\big)\,\psi_\circ\big((x,\phi_{e}(x))\big)=\iid_{\phi_e(x)}
	\end{equation}
	for all $e\in\Gamma^1$, $x\in\caG_{o(e)}^{(0)}$.
	Moreover, by definition one has
	\begin{equation}
	\begin{aligned}
	\psi_\circ\big((\alpha_e(x),\alpha_{\bar{e}}(x)\big)\, \psi_\circ(\alpha_{\bar{e}}(g))\, \psi_\circ\big((\alpha_{\bar{e}}(y),\alpha_e(y))\big)
	&=g_{e,\alpha_{\bar{e}}(x)}\,\alpha_{\bar{e}}(g)\, g_{e,\alpha_{\bar{e}}(y)}^{-1}\\
	&=\alpha_e(g)\\
	&=\psi_\circ(\alpha_{e}(g))
	\end{aligned}
	\end{equation}
	for all $e\in \Gamma^1$ and $g\in\caG_e$ with $r(g)=x$ and $s(g)=y$. 
\end{proof}


Since $\caG(\Gamma)$ in the desingularization of the action of $\caG$ on $F$ is a graph of groupoids,
one has that the fundamental groupoid 
$\pi_1(\caG(\Gamma))$ of $\caG(\Gamma)$ acts on 
the Bass-Serre forest $X_{\caG(\Gamma)}$ associated to 
$\caG(\Gamma)$ as we have seen is Subsection \ref{ss:fundgpd}.

\begin{prop}
	Let $\caG(\Gamma)$, $\pi_1(\caG(\Gamma))$ and $X_{\caG(\Gamma)}$ be as above.
	Let 
	\begin{equation}
	\Psi=(\Psi^0,\Psi^1)\colon X_{\caG(\Gamma)}\longrightarrow F
	\end{equation}
	be the mapping defined by
	\begin{equation}
	\begin{aligned}
	\Psi^0(p\caG_v[v])&= \mu^0\big(\,\psi (p),\, \sigma_v(\psi(s(p)))\,\big),\\
	\Psi^1(p\caH_{|e|}[e])&= \mu^1\big(\,\psi (p),\, \sigma_e(\psi(s(p)))\,\big),
	\end{aligned}
	\end{equation}
	where $\sigma_v$ and $\sigma_e$ are the partial section associated to $v\in\Gamma^0$
	and $e\in\Gamma^1$, respectively.
	Then $\Psi$ is a $\psi$-equivariant homomorphism of graphs, i.e., 
	\begin{equation}
	\begin{aligned}
	\Psi^0(\mu^0(p,w))&= \mu^0\big(\psi (p), \Psi^0(w)\big),\\
	\Psi^1(\mu^1(p,f))&= \mu^1\big(\psi (p), \Psi^1(f)\big)
	\end{aligned}
	\end{equation}
	for all $p\in\pi_1(\caG(\Gamma))$, $w\in X_{\caG(\Gamma)}^0$ and $f\in X_{\caG(\Gamma)}^1$.
\end{prop}
\begin{proof}
	We first prove that $\Psi$ is a homomorphism of graphs.
	Note that for $e\in\Upsilon_-$ and $p\caH_{|e|}[e]\in X_{\caG(\Gamma)}^1$, one has 
	\begin{align}
	t\big(\sigma_e(\psi(s(p))\big)&=\mu^0\big(\psi(g_{e,s(p)}), \sigma_{t(e)}(\psi(s(p))\big),\\
	o\big(\sigma_e(\psi(s(p))\big)&=\sigma_{o(e)}\big(\psi(s(p))\big).
	\end{align}
	For $e\in T^1\sqcup\Upsilon_+$ and $p\caH_{|e|}[e]\in X_{\caG(\Gamma)}^1$, one has 
	\begin{align}
	t\big(\sigma_e(\psi(s(p))\big)&= o\big(\sigma_{\bar{e}} (\psi(s(p)))\big)\notag\\
	&=\sigma_{o(\bar{e})}(\psi(s(p)))\notag\\
	&=\sigma_{t(e)}(\psi(s(p))),\notag\\
	o\big(\sigma_e(\psi(s(p))\big)&=t\big(\sigma_{\bar{e}}(\psi(s(p))\big)\notag\\
	&=\mu^0\big(\psi(g_{\bar{e}, s(p)}), \sigma_{t(\bar{e})} (\psi(s(p))\big)\notag\\
	&=\mu^0\big( \psi(g_{e, s(p)})^{-1} , \sigma_{o(e)}(\psi(s(p)))\big).\notag
	\end{align}
	Thus with the notation as in \eqref{eq:eps}, one obtains 
	\begin{align}
	\overline{\sigma_e(x)}&=\sigma_{\bar{e}}(x),\\
	t(\sigma_e(x))&=\mu^0\big( \psi(g_{e,x})^{1-\eps(e)}, \sigma_{t(e)}(x)\big), \\
	o(\sigma_e(x))&=\mu^0\big( \psi(g_{e},x)^{-\eps(e)}, \sigma_{o(e)}(x)\big),
	\end{align}
	for all $e\in\Gamma^1$, $x\in\dom\, \sigma_e$.
	This shows that $\Psi$ is a homomorphism of graphs which commutes with the 
	action of $\pi_1(\caG(\Gamma))$ on $X_{\caG(\Gamma)}$.
\end{proof}

One has the following structure theorem.

\begin{thm}
	\label{thm:structure}
	With the notation as above, $\psi$ is an isomorphism of groupoids 
	and $\Psi$ is an isomorphism of graphs.
\end{thm}

We need the following lemma for the proof of Theorem \ref{thm:structure}.
\begin{lem}
	\label{lem:injhomo}
	Let $\Psi=(\Psi^0,\Psi^1)\colon \Lambda\to \Delta$ be a homomorphism of graphs such that 
	\begin{itemize}
		\itemsep0em
		\item[(i)] $ \Lambda$ is connected;
		\item[(ii)] $\Delta$ is a tree;
		\item[(iii)] for all $u\in\Lambda^0$ the map $\Psi^1|_{\mathrm{st}_\Lambda (u)}\colon \mathrm{st}_\Lambda (u)\to \mathrm{st}_\Delta (\Psi^0(u))$ is injective.
	\end{itemize}
	Then $\Psi$ is injective.
\end{lem}
\begin{proof}
	Since $\Psi^1$ is injective by hypothesis, it suffices to show that $\Psi^0$ is injective.
	Suppose that there exist $v,w\in\Lambda^0$, $v\ne w$, such that $\Psi^0(v)=\Psi^0(w)$.
	Since $\Lambda$ is connected, there exists a path $p\in\caP_{v,w}(\Lambda)$ from
	$v$ to $w$.
	Since $\Delta$ is a tree, by (iii) one has that $\Psi_e(p)$ is a reduced path from 
	$\Psi^0(v)$ to $\Psi^0(v)$, which is a contradiction.
	Thus, $\Psi^0$ is injective.
\end{proof}

\begin{rem}
	\label{rem:fundom}
	We recall that by the construction of the desingularization $\frD(\caG,F)$,
	one has that $\Gamma$ is a fundamental domain
	for the action of $\caG$ on $F$.
	That is, for any $w\in F^0$ there exists $v\in\Gamma^0$ such that
	$w$ is in the saturation of $\image\,\sigma_v$, i.e., there exists
	$u\in\image\,\sigma_v\subseteq F^0$ such that $w$ is in the orbit of $u$.
	Similarly, for any $\eps\in F^1$ there exists $e\in\Gamma^1$ such that 
	$\eps$ is in the saturation of $\image\,\sigma_e$, i.e., there exists
	$f\in\image\,\sigma_e\subseteq F^1$ such that $\eps$ is in the orbit of $f$.
\end{rem}

We are now ready to prove the structure theorem.

\begin{proof}[Proof of Theorem \ref{thm:structure}]
	Let $\frD\big(\pi_1(\caG(\Gamma)), \,X_{\caG(\Gamma)}\big)$
	be the desingularization of the action of $\pi_1(\caG(\Gamma))$ on $X_{\caG(\Gamma)}$, 
	given by a graph of groupoids $\frG(\Delta)$ based on a graph $\Delta$ and
	a maximal subtree $\Lambda\subseteq\Delta$ such that $(\frG(\Lambda),\rho)$
	is a tree of representatives of the action of $\pi_1(\caG(\Gamma))$ on $X_{\caG(\Gamma)}$.
	For $v\in\Gamma^0$, consider the partial section $\sigma_v$ associated to $v$ 
	and the stabilizer
	\begin{equation*}
	\Stab_\caG (\sigma_v)=\big\{\, g\in\caG\mid \mu^0(g,\sigma_v(s(g)))=\sigma_v(r(g)),\, s(g),r(g)\in\dom(\sigma_v)\,\big\}.
	\end{equation*}
	Fix a unit $x\in\pi_1(\caG(\Gamma))^{(0)}$ and consider the vertex 
	$x\caG_v[v]\in X_{\caG(\Gamma)}^0$.
	Then there exists $w\in \Delta^0$ such that $x\caG_v[v]$ is in the saturation of $\image\,\sigma_w$.
	Consider the stabilizer
	\begin{equation*}
	\Stab_{\pi_1(\caG(\Gamma))} (\sigma_{w})=\big\{\,
	p\in\pi_1(\caG(\Gamma)) \mid \mu^0(p,\sigma_w(s(p)))=\sigma_w(r(p)),\, 
	s(p), r(p)\in\dom(\sigma_w)
	\,\big\}.
	\end{equation*}
	Then the map 
	\begin{equation}
	\label{eq:psirestr}
	\psi\RestrTo{\Stab_{\pi_1(\caG(\Gamma))}(\sigma_w)}\colon
	\Stab_{\pi_1(\caG(\Gamma))}(\sigma_w)\to \Stab_{\caG}(\sigma_v)
	\end{equation}
	is an isomorphism by construction.
	Then one has that the following diagram of graphs commutes:
	\begin{center}
		\begin{tikzcd}[column sep=normal, row sep=large]
			& X_{\caG(\Gamma)} \ar[rr, "\Psi"] 
			& 
			& F  
			&[2.0em]\\
			&  \caF^{\mathrm{pi}}(\Lambda) \ar[u, "\rho "] \ar[rr, "\widehat{\Psi}"]
			&
			&  \caF^{\mathrm{pi}}(T) \ar[u, "\chi "]
		\end{tikzcd}
	\end{center}
	where $\caF^\mathrm{pi}(\Lambda)$ and $\caF^\mathrm{pi}(T)$ are the 
	forests of partial isomorphisms induced by $\frG(\Lambda)$ and $\caG(T)$, respectively
	(cf. Remark \ref{rem:fpi}).
	In particular, $\widehat{\Psi}$ is an isomorphism of graphs.
	Hence one has that 
	\begin{equation}
	\mu^0\Big(\,\psi\big(\pi_1(\caG(\Gamma))\big),\, \image\,\sigma_v\,\Big)
	= \mu^0\big(\,\caG,\, \image\,\sigma_v\,\big)
	\end{equation}
	and, by \eqref{eq:psirestr}, the groupoid homomorphism 
	$\psi\colon\pi_1(\caG(\Gamma))\to \caG$ is surjective.
	Thus, since 
	\begin{equation}
	\begin{aligned}
	F^0&=\bigsqcup_{v\in\Gamma^0} \mu^0(\caG,\image\,\sigma_v) ,\\
	F^1&=\bigsqcup_{e\in\Gamma^1} \mu^1(\caG,\image\,\sigma_e),
	\end{aligned}
	\end{equation}
	by Remark \ref{rem:fundom}, one has that $\Psi$ is surjective.
	
	Note that $X_{\caG(\Gamma)}$ is fibered on $\pi_1(\caG(\Gamma))$ and each
	fiber corresponds to a connected component of $X_{\caG(\Gamma)}$, i.e.
	\begin{equation}
		X_{\caG(\Gamma)}=\bigsqcup_{y\in\pi_1(\caG(\Gamma))^{(0)}} yX_{\caG(\Gamma)},
	\end{equation}
	where $yX_{\caG(\Gamma)}$ denotes the fiber of $y\in\pi_1(\caG(\Gamma))^{(0)}$ 
	(cf. Remark \ref{rem:bsforest}).
	Similarly, the forest $F$ is fibered on $\caG^{(0)}$ and hence one has
	\begin{equation}
		F=\bigsqcup_{a\in\caG^{(0)}} aF.
	\end{equation}
	Since $\Psi$ is a $\psi$-equivariant homomorphism of graphs, 
	one has that $\Psi(yX_{\caG(\Gamma)})\subseteq F_{\psi(y)}$
	for all $y\in\pi_1(\caG(\Gamma))^{(0)}$.
	Then for $e\in\mathrm{st}_\Gamma(v)$ one has canonical bijections 
	\begin{center}
		\begin{tikzcd}[column sep=normal, row sep=large]
			& 
			& \bigsqcup_{e\in\mathrm{st}_\Gamma(v)} x\caG_v / \mathrm{im}(\alpha_{\bar{e}}) 
			\ar[dl, "\bigsqcup_{e\in\mathrm{st}_\Gamma(v)} a_e\quad" left] 
			\ar[dr, "\bigsqcup_{e\in\mathrm{st}_\Gamma(v)} b_e"]
			& 
			&[2.0em]\\
			&  \mathrm{st}_{X_{\caG(\Gamma)}}(x\caG_v[v]) \ar[rr, dashed, "\Psi^1|_{\mathrm{st}_{X_{\caG(\Gamma)}}(x\caG_v[v])}"] 
			&
			&  \mathrm{st}_{F}(\sigma_v(x))
		\end{tikzcd}
	\end{center}
	given by
	\begin{align}
	a_e(g)&=g\caH_{|e|}[e],\\
	b_e(g)&=\sigma_e(s(g)),
	\end{align}
	for $g\in\caG_v$ with $r(g)=x$.
	Hence $\Psi^1\RestrTo{\mathrm{st}_{X_{\caG(\Gamma)}}(x\caG_v[v])}$ is bijective for every 
	$x\caG_v[v]\in X_{\caG(\Gamma)}^0$.
	Thus, for any $y\in\pi_1(\caG(\Gamma))^{(0)}$ the restriction 
	\begin{equation}
	\Psi|_{yX_{\caG(\Gamma)}}\colon yX_{\caG(\Gamma)}\to F_{\psi(y)}
	\end{equation}
	is an injective homomorphism of graphs  
	by Lemma \ref{lem:injhomo}.
	Therefore, $\Psi$ is injective.
	Thus, $\Psi$ is an isomorphism of graphs.
	
	It remains to show that $\psi$ is injective.
	Let $N:=ker(\psi)=\{p\in\pi_1(\caG(\Gamma))\mid \psi(p)\in\caG^{(0)}\}$.
	For $w\in \Lambda^0$, 
	let $\pi_w=\Stab_{\pi_1(\caG(\Gamma))}(\sigma_w)$.
	One has that $N \cap \pi_w=\pi_w^{(0)}$ 
	by \eqref{eq:psirestr}.
	On the other hand, for $n\in N$ and $v\in X_{\caG(\Gamma)}^{(0)}$ 
	such that $s(n)=\tilde{\varphi}(v)$ and $v$ is in the saturation of $\image\,\sigma_w$,
	one has
	\begin{equation}
	\Psi^0\big(\mu^0(n,v)\big)=\mu^0\big(\psi(n), \Psi^0(v)\big)= \Psi^0(v)
	\end{equation}
	since $\Psi$ is a $\psi$-equivariant homomorphism of graphs.
	Since $\Psi$ is injective, one has that $\mu^0(n,v)=v$, i.e., 
	$n=\tilde{\varphi}(v)\in N\cap \pi_w=\pi_w^{(0)}$.
	Thus, $N=\pi_1(\caG(\Gamma))^{(0)}$ and hence 
	it is injective by Proposition \ref{prop:homoinjker}.
	This proves that $\psi$ is an isomorphism of groupoids.
\end{proof}

\begin{example}
	Let $\caG(\Gamma)$ be the graph of groupoids
	
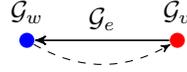
\begin{center}
		\begin{tikzpicture}[->,>=stealth']
		\node[style=circle,fill=red,inner sep=0pt, minimum size=2mm,label=above:$\caG_v$] (n1) at (3,0) {};
		\node[style=circle,fill=blue,inner sep=0pt, minimum size=2mm,label=above:$\caG_w$] (n2) at (1,0)  {};	
		\path[semithick]
		(n1) edge node [left][above] {$\caG_e$} (n2);	
		\path[dashed]
		(n2) edge [bend right] node [right][below] {} (n1);
		\end{tikzpicture}
	\end{center}
	where $\Gamma=\big(\{v,w\}, \{e\}\big)$ and the vertex and edge groupoids are given by 
	\begin{align}
	\caG_v&=\{x_1,x_2,a_v, a_v^{-1}\}, \quad s(a_v)=x_1, \, r(a_v)=x_2,\notag\\
	\caG_w&=\{y_1,y_2,a_w,a_w^{-1}\},\quad s(a_w)=y_1, \, r(a_w)=y_2,\notag\\
	\caG_e&=\caG_{\bar{e}}=\{z_1,z_2\}
	.\notag
	\end{align}
%
%
%
%
	The monomorphisms $\alpha_e\colon\caG_e\to\caG_w$ 
	and $\alpha_{\bar{e}}\colon\caG_e\to\caG_v$ are given by 
	\begin{equation*}
	\alpha_e(z_i)=y_i, 
	\quad
	\alpha_{\bar{e}}(z_i)=x_i,
	\notag
	\end{equation*}
	for $i=1,2$. Then one has 
	$\caH_e=\image(\alpha_e)=\caG_w^{(0)}$,
	$\caH_{\bar{e}}=\image(\alpha_{\bar{e}})=\caG_v^{(0)}$,
	and trasversals
	$\caT_e=\caG_w$ and 
	$\caT_{\bar{e}}=\caG_v$.
	We identify $x_i$ with $y_i$ for $i=1,2$ via the relation $\eqref{eq:relpi}$.
	Then the fundamental groupoid $\pi_1\big(\caG(\Gamma)\big)$ has 
	unit space $\pi_1\big(\caG(\Gamma)\big)^{(0)}=\{u_1,u_2\}$,
	generators 
	$
	\{\,a_v,\,a_v^{-1},\, a_w,\, a_w^{-1},\, 
	y_1ex_1, \,y_2ex_2, \,x_1\bar{e}y_1,\, x_2\bar{e}y_2\,\}
	$
	and defining relations
	$a_v^{-1} a_v=u_1$, $a_v a_v^{-1}=u_2$,
	$a_w^{-1} a_w=u_1$, $a_w a_w^{-1}=u_2$,
	together with relations
	\begin{align}
	&x_1 \bar{e} y_1=\iid_{u_1},\quad x_2 \bar{e} y_2=\iid_{u_2},\quad y_1 e x_1=\iid_{u_1},\quad 
	y_2e x_2=\iid_{u_2}; \tag{R1}\\
	&e x_1\bar{e}=y_1,\quad e x_2\bar{e}=y_2, \quad
	\bar{e} y_1 e=x_1,\quad \bar{e}y_2 e=x_2.\tag{R2}
	\end{align}
	Hence one has 
	\begin{align}
	\pi_1(\caG(\Gamma))=\{\,
	&u_1,u_2, \notag\\
	&a_v,a_v^{-1},a_w,a_w^{-1},\notag\\
	&a_v a_w^{-1},\, a_v^{-1} a_w,\,a_w a_v^{-1}, a_w^{-1} a_v\notag\\
	&a_v a_w^{-1} a_v,\, a_v^{-1} a_w a_v^{-1},\,a_w a_v^{-1} a_w, a_w^{-1} a_v a_w^{-1}\notag\\
	& \cdots \,\}.\notag
	\end{align}
	That is, the morphisms of $\pi_1(\caG(\Gamma))$ are the finite words 
	given by the alternation of composable red and blue arrows 
	where no two successive arrows of the same color occur,
	as shown in the subsequent diagram:
	
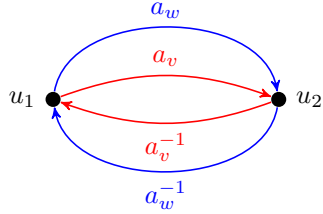
\begin{center}
		\begin{tikzpicture}[->,>=stealth']
		\node[style=circle,fill=black,inner sep=0pt, minimum size=2mm,label=left:{$u_1$}] (n1) at (-3,0) {};
		\node[style=circle,fill=black,inner sep=0pt, minimum size=2mm,label=right:{$u_2$}] (n2) at (0,0)  {};
		
		\path[red,semithick]
		(n1) edge [bend left=20] node [above] {$a_v$} (n2)
		(n2) edge [bend left=20] node [below] {$a_v^{-1}$} (n1);	
		
		\path[blue,semithick]
		(n1) edge [bend left=80] node [above] {$a_w$} (n2)
		(n2) edge [bend left=80] node [below] {$a_w^{-1}$} (n1);	
		\end{tikzpicture}
	\end{center}
	The Bass-Serre forest $X_{\caG(\Gamma)}$ is defined as follows:
	\begin{align}
	X_{\caG(\Gamma)}^0&=\{\,p\caG_v[v]\mid p\in\pi_1(\caG(\Gamma))\,\}\sqcup
	\{\,p\caG_w[w]\mid p\in\pi_1(\caG(\Gamma))\,\},\notag\\
	X_{\caG(\Gamma)}^1&=\{\,p\caH_{|e|}[e]\mid p\in\pi_1(\caG(\Gamma))\,\},\notag
	\end{align}
	where $p\caG_v=\{pg\mid g\in\caG_v, s(p)=r(g)\}$ (see Definition \ref{def:BSforest}).
	Then $X_{\caG(\Gamma)}$ is a bipartite graph, since the vertices of $X_{\caG(\Gamma)}$ are
	naturally partitioned into two disjoint sets and each edge has initial vertex
	in one of these sets and terminal vertex in the other set.	
	The Bass-Serre forest $X_{\caG(\Gamma)}$ is the union of the graphs $u X_{\caG(\Gamma)}$, 
	for $u\in\pi_1(\caG)^{(0)}$, i.e.,
	\begin{equation*}
	X_{\caG(\Gamma)}=u_1 X_{\caG(\Gamma)}\sqcup u_2 X_{\caG(\Gamma)}:
	\end{equation*}
	
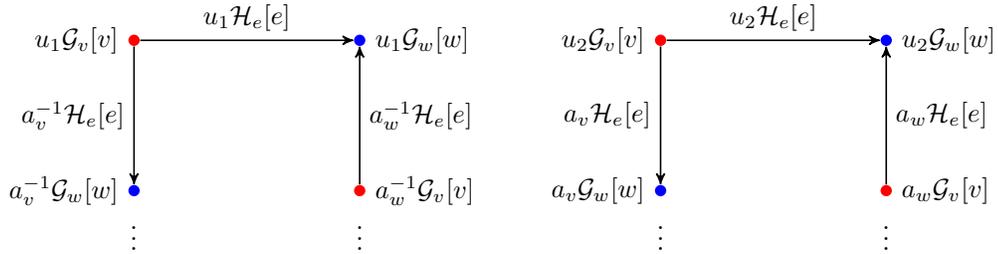
	\begin{center}
		\begin{tikzpicture}[->,>=stealth']
		\node[style=circle,fill=red,inner sep=0pt, minimum size=1.5mm,label=left:{$u_1\caG_v[v]$}] (n1) at (-4,0) {};
		\node[style=circle,fill=blue,inner sep=0pt, minimum size=1.5mm,label=right:{$u_1\caG_w[w]$}] (n2) at (-1,0) {};
		\node[style=circle,fill=blue,inner sep=0pt, minimum size=1.5mm,label=left:{$a_v^{-1}\caG_w[w]$}] (n3) at (-4,-2)  {};
		\node[style=circle,fill=red,inner sep=0pt, minimum size=1.5mm,label=right:{$a_w^{-1}\caG_v[v]$}] (n4) at (-1,-2)  {};
		\node[style=circle,fill=red,inner sep=0pt, minimum size=1.5mm,label=left:{$u_2\caG_v[v]$}] (n5) at (3,0)  {};	
		\node[style=circle,fill=blue,inner sep=0pt, minimum size=1.5mm,label=right:{$u_2\caG_w[w]$}] (n6) at (6,0)  {};	
		\node[style=circle,fill=blue,inner sep=0pt, minimum size=1.5mm,label=left:{$a_v\caG_w[w]$}] (n7) at (3,-2)  {};
		\node[style=circle,fill=red,inner sep=0pt, minimum size=1.5mm,label=right:{$a_w\caG_v[v]$}] (n8) at (6,-2)  {};
		\node[style=circle,inner sep=0pt, minimum size=1.5mm,label=:{$\vdots$}] (n9) at (-4,-3)  {};
		\node[style=circle,inner sep=0pt, minimum size=1.5mm,label=:{$\vdots$}] (n9) at (-1,-3)  {};
		\node[style=circle,inner sep=0pt, minimum size=1.5mm,label=:{$\vdots$}] (n9) at (3,-3)  {};
		\node[style=circle,inner sep=0pt, minimum size=1.5mm,label=:{$\vdots$}] (n9) at (6,-3)  {};
		
		\path[semithick]
		(n1) edge node [right] [above] {$u_1\caH_e[e]$} (n2)
		(n1) edge node [left][left] {$a_v^{-1}\caH_e[e]$} (n3)
		(n4) edge node [right] {$a_w^{-1}\caH_e[e]$} (n2)
		(n5) edge node [above] {$u_2\caH_e[e]$} (n6)
		(n5) edge node [left] {$a_v\caH_e[e]$} (n7)
		(n8) edge node [left][right] {$a_w\caH_e[e]$} (n6);
		\end{tikzpicture}
	\end{center}
	The fundamental groupoid $\pi_1(\caG(\Gamma))$ acts on $X_{\caG(\Gamma)}$ by left multiplication with momentum map $\varphi\colon X_{\caG(\Gamma)}^0\to\pi_1(\caG)^{(0)}$ given by $\varphi(q\caG_v)=r(q)$.
	Hence one has that 
	$X_{\caG(\Gamma)}=\varphi^{-1}(u_1)\sqcup\varphi^{-1}(u_2)$.
	There are two orbits of the action of $\pi_1(\caG(\Gamma))$ on $X_{\caG(\Gamma)}^0$:
	one is given by the collection of the red vertices and the other is given by the collection of the
	blue vertices, i.e., 
	\begin{equation*}
	X_{\caG(\Gamma)}^0=\mu^0\big(\pi_1(\caG(\Gamma)), u_1\caG_v \big)\sqcup
	\mu^0\big(\pi_1(\caG(\Gamma)), u_1\caG_w \big).
	\end{equation*}
	We enumerate the fibers of $u_1$ and $u_2$ (see the proof of Theorem \ref{thm:desingular}), i.e., we enumerate the vertices of 
	$u_1 X_{\caG(\Gamma)}$ and $u_2 X_\caG(\Gamma)$ as follows:
	
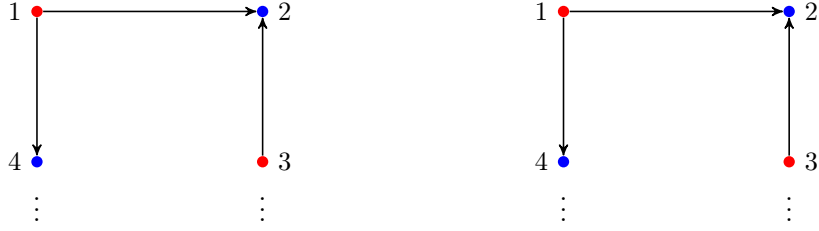
\begin{center}
		\begin{tikzpicture}[->,>=stealth']
		\node[style=circle,fill=red,inner sep=0pt, minimum size=1.5mm,label=left:{1}] (n1) at (-4,0) {};
		\node[style=circle,fill=blue,inner sep=0pt, minimum size=1.5mm,label=right:{2}] (n2) at (-1,0) {};
		\node[style=circle,fill=blue,inner sep=0pt, minimum size=1.5mm,label=left:{4}] (n3) at (-4,-2)  {};
		\node[style=circle,fill=red,inner sep=0pt, minimum size=1.5mm,label=right:{3}] (n4) at (-1,-2)  {};
		\node[style=circle,fill=red,inner sep=0pt, minimum size=1.5mm,label=left:{1}] (n5) at (3,0)  {};	
		\node[style=circle,fill=blue,inner sep=0pt, minimum size=1.5mm,label=right:{2}] (n6) at (6,0)  {};	
		\node[style=circle,fill=blue,inner sep=0pt, minimum size=1.5mm,label=left:{4}] (n7) at (3,-2)  {};
		\node[style=circle,fill=red,inner sep=0pt, minimum size=1.5mm,label=right:{3}] (n8) at (6,-2)  {};
		\node[style=circle,inner sep=0pt, minimum size=1.5mm,label=:{$\vdots$}] (n9) at (-4,-3)  {};
		\node[style=circle,inner sep=0pt, minimum size=1.5mm,label=:{$\vdots$}] (n9) at (-1,-3)  {};
		\node[style=circle,inner sep=0pt, minimum size=1.5mm,label=:{$\vdots$}] (n9) at (3,-3)  {};
		\node[style=circle,inner sep=0pt, minimum size=1.5mm,label=:{$\vdots$}] (n9) at (6,-3)  {};
		
		\path[semithick]
		(n1) edge node [right] [above] {} (n2)
		(n1) edge node [left][left] {} (n3)
		(n4) edge node [right] {} (n2)
		(n5) edge node [above] {} (n6)
		(n5) edge node [left] {} (n7)
		(n8) edge node [left][right] {} (n6);
		\end{tikzpicture}
	\end{center}
	Such enumeration gives two partial sections of vertices $\sigma_1,\sigma_2\colon \pi_1(\caG(\Gamma))^{(0)}\to X_{\caG(\Gamma)}^{0}$ and a partial section of edges $\sigma_\be\colon \pi_1(\caG(\Gamma))^{(0)}\to X_{\caG(\Gamma)}^1$
	defined by 
	\begin{align}
	\sigma_1(u_i)&=u_i\caG_v[v], \notag\\
	\sigma_2(u_i)&=u_i\caG_w[w],\notag\\
	\sigma_\be(u_i)&=u_i\caH_e[e],\notag
	\end{align}
	for $i=1,2$.
	Hence, we have 
	a tree of representatives $\frG(T)$ based on the segment tree 
	$T=\big(\{v_1,v_2\}, \{e,\bar{e}\}\big)$ as follows
	
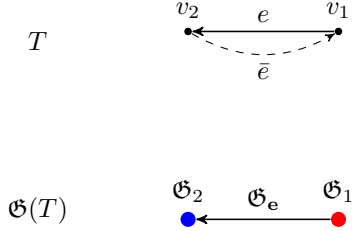
\begin{center}
		\begin{tikzpicture}[->,>=stealth']
		\node[label={$T$}] (n) at (-1,2) {};
		\node[label={$\frG(T)$}] (n) at (-1,-0.3) {};
		
		\node[style=circle,fill,inner sep=0pt, minimum size=1mm,label=above:{$v_1$}] (n1) at (3,2.5) {};
		\node[style=circle,fill,inner sep=0pt, minimum size=1mm,label=above:{$v_2$}] (n2) at (1,2.5)  {};
		\path[semithick]
		(n1) edge node [left][above] {$e$} (n2);	
		\path[dashed]
		(n2) edge [bend right] node [right][below] {$\bar{e}$} (n1);
		
		\node[style=circle,fill=red,inner sep=0pt, minimum size=2mm,label=above:{$\frG_1$}] (n3) at (3,0) {};
		\node[style=circle,fill=blue,inner sep=0pt, minimum size=2mm,label=above:{$\frG_2$}] (n4) at (1,0)  {};	
		\path[semithick]
		(n3) edge node [left][above] {$\frG_\be$} (n4);
		\end{tikzpicture}
	\end{center}
	where $\frG_1=\Stab_{\pi_1(\caG(\Gamma))}(\sigma_1)$, 
	$\frG_2=\Stab_{\pi_1(\caG(\Gamma))}(\sigma_2)$,
	and $\frG_\be=\Stab_{\pi_1(\caG(\Gamma))}(\sigma_\be)$.
	In particular, $\frG(T)$ is a desingularization of the action of $\pi_1(\caG(\Gamma))$
	on $X_{\caG(\Gamma)}$.
\end{example}


\bibliography{bsgroupoids}
\bibliographystyle{amsplain}

\end{document}